\documentclass[10pt]{amsart}

\usepackage{mathabx}
\usepackage{graphicx}
\usepackage{pgf,tikz}
\AtBeginDocument{}
\usetikzlibrary{decorations.pathreplacing}

\usepackage{caption}

\usepackage{amssymb}
\usepackage{stmaryrd}
\usepackage{mathrsfs}
\usepackage[all]{xy}

\makeatletter
  \def\@wrindex#1{%
    \protected@write\@indexfile{}%
      {\string\indexentry{#1}{ \S\thesubsection (p.\thepage)}}
    \endgroup
  \@esphack}

\makeatother

\makeindex

\usepackage{rotating}

\usepackage{amscd}
\usepackage{epsfig}

\author{Benjamin Enriquez et Florence Lecomte}

\address{Institut de Recherche Math\'{e}matique Avanc\'{e}e, UMR 7501, 
Universit\'{e} de Strasbourg et CNRS, 7
rue Ren\'{e} Descartes, 67000 Strasbourg, France}
\email{enriquez@math.unistra.fr}

\email{lecomte@math.unistra.fr}

\date{3 janvier 2023}
\newtheorem{thm}{Théorème}[section]
\newtheorem{lem}[thm]{Lemme}
\newtheorem{lemdef}[thm]{Lemme-Définition}
\newtheorem{cor}[thm]{Corollaire}
\newtheorem{prop}[thm]{Proposition}

{\theoremstyle{definition} \newtheorem{rem}[thm]{Remarque}}
{\theoremstyle{definition} \newtheorem{defn}[thm]{Définition}}

{\theoremstyle{remark} }

\numberwithin{equation}{subsection}
\numberwithin{figure}{section}

\begin{document}

\baselineskip 16pt 

\title[Transformations naturelles reliant foncteurs homotopique et 
homologique]{Transformations naturelles reliant foncteurs d'homotopie 
et d'homologie singulière}

\begin{abstract}
La catégorie des espaces topologiques avec deux points marqués est munie de deux familles $\mathbf F_n$ et $\mathbf H_n$, 
indexées par un entier 
$n\geq0$, de foncteurs vers la catégorie des groupes abéliens, la première associant à l'objet $(X,x,y)$ le quotient de 
$\mathbb Z\pi_1(X,x,y)$ par un sous-groupe abélien associé à la $n+1$-ième puissance de d'idéal d'augmentation de l'algèbre de 
groupe $\mathbb Z\pi_1(X,x)$, la seconde associant au même objet le $n$-ième groupe d'homologie singulière relative de $X^n$ par rapport à un 
sous-espace défini en termes de diagonales partielles. Nous construisons une famille de transformations naturelles $\nu_n : \mathbf F_n\to 
\mathbf H_n$. Nous identifions la transformation naturelle obtenue par restriction de $\nu_n$ à la sous-catégorie des variétés algébriques et 
tensorisation avec $\mathbb Q$ avec l'équivalence naturelle due à Beilinson. 
\end{abstract}

\bibliographystyle{amsalpha+}
\maketitle
{\footnotesize \setcounter{tocdepth}{3}\tableofcontents}

\newpage

\section*{Introduction}

Pour $X$ un espace topologique connexe et $a,b \in X$, on note $\pi_1(a,b)$ l'ensemble des classes de chemins reliant $a$ à $b$. 
Le $\mathbb Z$-module $\mathbb Z\pi_1(a,b)$ est alors un module à droite sous l'action de l'algèbre du groupe $\pi_1(a):=\pi_1(a,a)$, 
et on définit, pour $n\geq 0$, le $\mathbb Z$-module $\mathbf F_n(X,a,b)$ comme son quotient par le sous-module 
$\mathbb Z\pi_1(a,b)\cdot(\mathbb Z\pi_1(a))_+^{n+1}$ engendré par l'action de la $n+1$-ème puissance de l'idéal d'augmentation 
de l'algèbre de groupe de $\pi_1(a)$. 

Si $X$ est une variété différentiable connexe, ayant le type d’homotopie d’un CW-complexe fini, on dispose 
d'une interprétation cohomologique de $\mathrm{Hom}_{\mathbb Z}(\mathbf F_n(X,a,b),\mathbb Q)$, sous la forme
d'un isomorphisme
de $\mathbb Q$-espaces vectoriels
\begin{equation}\label{iso:beil}
    \mathrm H^n(X^n,Y^{(n)}_{ba};\mathbb Q)\stackrel{\sim}{\to} \mathrm{Hom}_{\mathbb Z}(\mathbf F_n(X,a,b),\mathbb Q), 
\end{equation}
où $Y^{(n)}_{ba}$ est la partie de $X^n$ définie par 
\begin{equation}\label{ref:Y:2212}
Y^{(n)}_{ba}:=\cup_{i=0}^n Y_{ba,i}^{(n)}
\end{equation}
avec  $Y_{ba,i}^{(n)}=\{(x_1,\ldots,x_n)\in X^n|x_i=x_{i+1}\}$ avec $x_0=b$ et $x_{n+1}=a$, 
et $\mathrm H^\bullet(-,-;\mathbb Q)$ désigne la cohomologie singulière relative à coefficients 
dans $\mathbb Q$ (travail de Beilinson, rédigé dans \cite{DG,BGF}). La construction de cet isomorphisme repose sur des techniques faisceautiques : 
précisément, on construit un morphisme $_b\tilde{\mathcal K}_a\to _b\mathcal K_a$ de complexes de faisceaux sur $X$ et un 
isomorphisme $iso_{\mathrm{BGF}}^{ba} : \mathbb H^n(X^n,_b \tilde{\mathcal K}_a\langle n\rangle)
\to \mathrm H^n(X^n,Y^{(n)}_{ba};\mathbb Q)$, où $\mathbb H(-,-)$ désigne l'hypercohomologie des complexes de faisceaux 
(\cite{BGF}, lemme 3.281) ; \eqref{iso:beil} est alors construit comme une composition 
$$
\mathrm H^n(X^n,Y^{(n)}_{ba};\mathbb Q)
\stackrel{(iso_{\mathrm{BGF}}^{ba})^{-1}}{\to}\mathbb H^n(X^n,_b \tilde{\mathcal K}_a\langle n\rangle)
\to\mathbb H^n(X^n,_b\mathcal K_a\langle n\rangle)
\to\mathrm{Hom}_{\mathbb Z}(\mathbf F_n(X,a,b),\mathbb Q). 
$$

La nature topologique de la source et du but de l'application \eqref{iso:beil} suggère la possibilité 
d'une construction topologique de cette application. Le but de ce travail est de fournir une telle construction, 
dans le cadre plus général où $X$ est un espace topologique, et en travaillant sur $\mathbb Z$. 
Plus précisément, nous contruisons un morphisme de $\mathbb Z$-modules  
$$
\mathbf F_n(X,a,b)\to\mathrm H_n(X^n,Y^{(n)}_{ab}), 
$$
où $\mathrm H_n(-,-)$ est l'homologie singulière relative (à coefficients dans $\mathbb Z$), cf. théorème \ref{thm:ppal}. 
Nous montrons, dans le cas où $X$ est une variété différentiable comme ci-dessus, la compatibilité de ce morphisme 
avec \eqref{iso:beil} et l'isomorphisme $\mathrm H_n(X^n,Y^{(n)}_{ab})\to \mathrm H_n(X^n,Y^{(n)}_{ba})$ induit par 
l'automorphisme de $X^n$ donné par $(x_1,\ldots,x_n)\mapsto(x_n,\ldots,x_1)$ (cf. proposition \ref{prop:beil:2012}). 

Ce travail est organisé comme suit : la section \ref{sect:rappels} contient des rappels sur l'homologie singulière~; la section principale 
est la section \ref{sect:ppale}, qui a pour objectif la démonstration du théorème \ref{thm:ppal} ; la section \ref{sect:3:2512} établit le 
lien de ce 
résultat avec l'isomorphisme \eqref{iso:beil} de Beilinson (proposition \ref{prop:beil:2012}) ; en section \ref{sect:4:2512}, on étudie
l'aspect fonctoriel de l'application construite dans le théorème \ref{thm:ppal}.  

\section{Rappels}\label{sect:rappels}

En section \ref{sect:1:1:2512}, on rappelle la construction de l'homologie singulière, et en section \ref{sect:1:2:2512}, 
celles de l'homologie et la cohomologie relatives.   

\subsection{Espaces topologiques et homologie singulière}\label{sect:1:1:2512}

On note $\mathbf{Top}$ la catégorie des espaces topologiques ; si $X,Y$ sont deux espaces topologiques, on note ainsi 
$\mathbf{Top}(X,Y)$ l'ensemble des applications continues $X\to Y$. 

Pour $n\geq 0$, on note $\Delta^n$ le simplexe donné par $\Delta^n:=\{(t_1,\ldots,t_n)\in\mathbb R^n|0\leq t_1\leq\ldots\leq t_n\leq 1\}$. 
Si $X$ est un espace topologique, on pose $C_n(X):=\mathbb Z\mathbf{Top}(\Delta^n,X)$ pour $n\geq0$, ainsi que $C_{-1}(X)=0$. 
Un sous-ensemble $Y$ de $X$ est naturellement muni de la topologie induite, on note $Y_X$ (ou simplement $Y$ s'il n'y a pas de risque 
de confusion) l'espace topologique correspondant. On a alors $\mathbf{Top}(\Delta^n,Y_X)=\{f\in \mathbf{Top}(\Delta^n,X)|
f(\Delta^n)\subset Y\}$ et $C_n(Y_X):=\mathbb Z\mathbf{Top}(\Delta^n,Y_X)\subset C_n(X)$ pour $n\geq0$. 

Pour $X$ un espace topologique et $n\geq0$, on note $\partial_{n,n-1}^* : C_n(X)\to C_{n-1}(X)$ la différentielle singulière. L'homologie du complexe 
$(C_\bullet(X),\partial^*)$ est alors l'homologie singulière $\mathrm H_\bullet(X)$. 

On note 
$\mathbf{AbGr}$ celle des groupes abéliens $\mathbb Z$-gradués. 
L'homologie singulière définit un foncteur $\mathrm H_\bullet : \mathbf{Top}\to\mathbf{AbGr}$. 

\subsection{Paires d'espaces topologiques et (co)homologie singulière relative}\label{sect:1:2:2512}

Si $X$ est un espace topologique et $Y$ est un sous-ensemble de $X$, alors $(C_\bullet(Y_X),\partial^*)$ est un sous-complexe de 
$(C_\bullet(X),\partial^*)$, et l'homologie relative $\mathrm H_\bullet(X,Y)$ est l'homologie du complexe quotient 
$$
(C_\bullet(X)/C_\bullet(Y_X),\partial^*) ;  
$$
c'est un groupe abélien gradué. La cohomologie relative $\mathrm H^\bullet(X,Y)$ est celle du sous-complexe 
$C_\bullet(Y_X)^\perp$ du complexe $\mathrm{Hom}_{\mathbb Z}(C_\bullet(X),\mathbb Z)$, muni de la différentielle
duale de $\partial^*$ ; on dispose d'un couplage $\mathrm H_\bullet(X,Y)\otimes\mathrm H^\bullet(X,Y)\to\mathbb Z$. Alors 
$\mathrm H^\bullet(X,Y;\mathbb Q)=\mathrm H^\bullet(X,Y)\otimes\mathbb Q$ est la cohomologie 
du sous-complexe $C_\bullet(Y_X)^\perp\otimes\mathbb Q$ de $\mathrm{Hom}_{\mathbb Z}(C_\bullet(X),\mathbb Q)$. 

Soit $\mathbf{Paires}$ la catégorie des paires d'espaces topologiques. Les  objets de $\mathbf{Paires}$ sont les couples $(X,Y)$, 
avec $X$ espace topologique et $Y$ sous-ensemble de $X$. On a 
$\mathbf{Paires}((A,B),(A',B')):=\{f\in \mathbf{Ens}(A,A')|f$ est continue et $f(B)\subset B'\}$. 
Alors $(X,Y)\mapsto \mathrm H_\bullet(X,Y)$ définit un foncteur 
$\mathrm{H}_\bullet : \mathbf{Paires}\to\mathbf{AbGr}$. 
On note $f_*$ le morphisme $\mathrm{H}_\bullet(A,B)\to \mathrm{H}_\bullet(A',B')$ dans $\mathbf{AbGr}$ associé à un morphisme 
$f:(A,B)\to(A',B')$ dans $\mathbf{Paires}$ (noté $\mathrm{H}_\bullet(f)$ dans \cite{Ha}, pp. 108, 124). 


L'application identité $\mathrm{id}_{\Delta^n}$ définit un élément de $C_n(\Delta^n)$, dont l'image dans $C_n(\Delta^n)/C_n(\partial\Delta^n)$ 
est un cycle pour le bord relatif, et définit donc  un élément du groupe d'homologie relative $[\mathrm{id}_{\Delta^n}]\in 
\mathrm{H}_n(\Delta^n,\partial\Delta^n)$. 

\section{Une identité en homologie relative}\label{sect:ppale}

Le but de cette section est la démonstration du théorème \ref{thm:ppal}. Ce résultat est formulé en section \ref{sect:res:ppal}, 
et sa première partie (théorème \ref{thm:ppal}(a)) est démontrée en section \ref{sect:2:2:2512}. Le reste de la section est consacré à la démonstration
de sa deuxième partie (théorème \ref{thm:ppal}(b)).  La section \ref{sect:2:3:2512} contient des résultats combinatoires, et la section \ref{sect:2:4:2512}
l'application de ses résultats à la construction d'endomorphismes $\mathrm{div}_\bullet^k$ des complexes de chaînes singulières satisfaisant 
la relation d'homotopie \eqref{132:2809}. Cette relation est appliquée en section \ref{1214:2111} à la démonstration du théorème \ref{thm:ppal}(b). 

\subsection{Matériel de base et résultat principal}\label{sect:res:ppal}

Dans la section \ref{sect:res:ppal}, on fixe un espace topologique $X$, des éléments $a,b\in X$, et $n\geq1$. On définit la partie 
$Y^{(n)}_{ab}\subset X^n$ par \eqref{ref:Y:2212}. 
On note $\mathrm{Chem}(a,b)$ l'ensemble des applications continues $\tilde\gamma : [0,1]\to X$ telles que 
$\tilde\gamma(0)=a$ et $\tilde\gamma(1)=b$. 

\begin{defn}\label{def:2:1:2311}
Pour $\tilde\gamma\in\mathrm{Chem}(a,b)$, on note $\tilde\gamma^{(n)}\in\mathbf{Top}(\Delta^n,X^n)$ l'application composée 
$\Delta^n\stackrel{\mathrm{can}_n}{\to}[0,1]^n\stackrel{\tilde\gamma^n}{\to}X^n$ où $\mathrm{can}_n : \Delta^n\to[0,1]^n$ est l'injection canonique. 
\end{defn}

\begin{lem}\label{lem:1:2}
$\partial_{n,n-1}^*(\tilde\gamma^{(n)})\in C_{n-1}(Y^{(n)}_{ab})$.  
\end{lem}

\begin{proof}
On a $\partial_{n,n-1}^*(\tilde\gamma^{(n)})=\sum_{i=0}^n(-1)^i\tilde\gamma^{(n)}\circ\partial_i^{n}$, où pour
$i\in[\![0,n]\!]$ l'application $\partial_i^{n} : \Delta^{n-1}\to\Delta^n$ est donnée par $(t_1,\ldots,t_{n-1})
\mapsto (t_1,\ldots,t_i,t_i,\ldots,t_{n-1})$ (avec par convention $t_0=0$, $t_n=1$). On vérifie que 
$\tilde\gamma^{(n)}\circ\partial_i^{n}=\partial_i^{n,X}\circ\tilde\gamma^{(n-1)}$, avec 
$\partial_i^{n,X} : X^{n-1}\to X^n$ l'application donnée par $(x_1,\ldots,x_{n-1})
\mapsto (x_1,\ldots,x_i,x_i,\ldots,x_{n-1})$ (avec par convention $x_0=a$, $x_n=b$), donc 
$\partial_{n,n-1}^*(\tilde\gamma^{(n)})=\sum_{i=0}^n(-1)^i\partial_i^{n,X}\circ\tilde\gamma^{(n-1)}$. 
On a pour tout $i\in[\![0,n]\!]$ les relations $\partial_i^{n,X}(X^{n-1})=Y_{ab,i}^{(n)}\subset Y^{(n)}_{ab}$
qui impliquent la relation annoncée.  
\end{proof}

On rappelle que $C_\bullet(Y^{(n)}_{ab})$ est un sous-complexe de $C_\bullet(X^n)$, et que l'homologie du complexe 
quotient $C_\bullet(X^n)/C_\bullet(Y^{(n)}_{ab})$ est l'homologie relative $\mathrm{H}_\bullet(X^n,Y^{(n)}_{ab})$.  

Il suit du lemme \ref{lem:1:2} que la classe de $\tilde\gamma^{(n)}$ dans $C_n(X^n)/C_n(Y^{(n)}_{ab})$ est un cycle du complexe quotient, et
définit donc une classe $[\tilde\gamma^{(n)}]\in \mathrm{H}_n(X^n,Y^{(n)}_{ab})$. On note $\pi_1(a,b)$ le quotient de 
$\mathrm{Chem}(a,b)$ par la relation d'équivalence donnée par l'homotopie entre deux chemins. 

\begin{thm}\label{thm:ppal}
(a) Il existe une unique application $\pi_1(a,b)\to \mathrm{H}_n(X^n,Y^{(n)}_{ab})$, $\gamma\mapsto F_n(\gamma)$ telle que l'application 
$\mathrm{Chem}(a,b)\to \mathrm{H}_n(X^n,Y^{(n)}_{ab})$, $\tilde\gamma\mapsto[\tilde\gamma^{(n)}]$ admette une factorisation 
$\mathrm{Chem}(a,b)\to \pi_1(a,b)\to\mathrm{H}_n(X^n,Y^{(n)}_{ab})$. On a donc pour $\tilde\gamma\in \mathrm{Chem}(a,b)$
et on notant $\tilde\gamma\to[\tilde\gamma]$ l'application canonique $\mathrm{Chem}(a,b)\to\pi_1(a,b)$, 
\begin{equation}\label{***:2809}
    [\tilde\gamma^{(n)}]=F_n([\tilde\gamma]). 
\end{equation}

(b) Pour $\alpha_0,\ldots,\alpha_n\in \pi_1(a)$, et $I\subset [\![0,n]\!]$, on pose $\prod_{i\in I}\alpha_i$ le produit
$\alpha_{i(1)}\cdots\alpha_{i(|I|)}$, où $i$ est l'unique bijection croissante $[\![1,|I|]\!]\to I$. Alors 
$$
\sum_{I\subset[\![0,n]\!]}(-1)^{|I|}F_n(\gamma\cdot \prod_{i\in I}\alpha_i)=0
$$
(égalité dans $\mathrm{H}_n(X^n,Y^{(n)}_{ab})$). On a donc factorisation de $F_n$ en une application linéaire 
$$
\overline F^{(n)}_{ab} : \mathbb Z\pi_1(a,b)/(\mathbb Z\pi_1(a,b)(\mathbb Z\pi_1(a))_+^{n+1})\to \mathrm{H}_n(X^n,Y^{(n)}_{ab}), 
$$
où $(\mathbb Z\pi_1(a))_+$ est l'idéal d'augmentation de $\mathbb Z\pi_1(a)$. 
\end{thm}

\subsection{Démonstration de (a) du théorème \ref{thm:ppal}}\label{sect:2:2:2512}

\begin{lem}\label{lem:def:0803}
(a) Pour $\tilde\gamma\in\mathrm{Chem}(a,b)$, $\tilde \gamma^{(n)}$ induit un morphisme 
$$
\tilde\gamma^{(n),\mathbf{Paires}
} : (\Delta^n,\partial\Delta^n)\to(X^n,Y^{(n)}_{ab})
$$
dans $\mathbf{Paires}$. 

(b) Pour $\tilde\gamma\in\mathrm{Chem}(a,b)$, on a $[\tilde\gamma^{(n)}]
=(\tilde\gamma^{(n),\mathbf{Paires}})_*([\mathrm{id}_{\Delta^n}])\in 
\mathrm{H}_n(X^n,Y^{(n)}_{ab})$. 
\end{lem}

\begin{proof}
(a) suit de la démonstration du lemme \ref{lem:1:2}. (b) suit de ce que l'élément 
$\tilde\gamma^{(n)}\in C_n(X^n)$ est l'image par le morphisme $\Delta^n\to X^n$ dans $\mathbf{Top}$ induit par  
$\tilde\gamma^{(n)}$ de $\mathrm{id}_{\Delta^n} \in C_n(\Delta^n)$.  
\end{proof}

Le théorème \ref{thm:ppal},(a) suit alors du lemme suivant : 
\begin{lem}
L'application $\mathrm{Chem}(\underline X)\to \mathrm{H}_n(X^n,Y_{ab}^{(n)})$, 
$\tilde\gamma\mapsto[\tilde\gamma^{(n)}]$ est invariante par homotopie. 
\end{lem}

\begin{proof}
Soit $\tilde\gamma,\tilde\gamma'\in\mathrm{Chem}(\underline X)$. Une homotopie entre $\tilde\gamma$ et 
$\tilde\gamma'$ produit une homotopie entre les morphismes de paires $(\Delta^n,\partial\Delta^n)\to(X^n,Y^{(n)}_{ab})$ donnés par 
$\tilde\gamma^{(n),\mathbf{Paires}}$ et $(\tilde\gamma')^{(n),\mathbf{Paires}}$. Par l'invariance homotopique de l'homologie relative 
(Proposition 13.14 dans \cite{Gr}), les morphismes associés $\mathrm{H}_n(\Delta^n,\partial\Delta^n)\to \mathrm{H}_n(X^n,Y^{(n)}_{ab})$ 
induits en homologie, à savoir $\mathrm{H}_n(\tilde\gamma^{(n),\mathbf{Paires}})$ et $\mathrm{H}_n((\tilde\gamma')^{(n),\mathbf{Paires}})$ 
sont égaux. Les images qu'ils donnent à $[\mathrm{id}_{\Delta^n}]$ sont donc égales, et le lemme \ref{lem:def:0803}(b) implique alors 
$[\tilde\gamma^{(n)}]=[\tilde\gamma^{\prime(n)}]$.
\end{proof}

\subsection{Constructions combinatoires}\label{sect:2:3:2512}

Le but de cette sous-section est la construction, pour tout couple $(n,k)$ avec $n,k\geq 1$ : (a)     
d'un ensemble $\mathbf{Aff}(\mathbb R^{n-1},\mathbb R^n)$, d'une application 
$(\underline f,\mathrm{sgn}) : \mathbb Z^n\times\mathfrak S_n \times[\![0,n]\!]\to 
\mathbf{Aff}(\mathbb R^{n-1},\mathbb R^{n-1})\times\{\pm1\}$ et d'une involution 
$\mathrm{invol}$ de $\mathbb Z^n\times\mathfrak S_n \times[\![0,n]\!]$, telle que le diagramme \eqref{diag:a:0412}  
commute  ; (b) d'applications $(\underline{\tilde f},\tilde{\mathrm{sgn}}) : \mathbb Z^{n-1}\times\mathfrak S_{n-1} \times[\![0,n]\!]
\to \mathbf{Aff}(\mathbb R^{n-1},\mathbb R^{n-1})\times\{\pm1\}$ et $\underline{\mathrm{bij}} : 
\mathbb Z^{n-1}\times\mathfrak S_{n-1} \times[\![0,n]\!]\to\mathbb Z^n\times\mathfrak S_n \times[\![0,n]\!]$
telles que le diagramme \eqref{diag:b:0412} commute 
(c) d'un sous-ensemble $\mathrm{Ens}_n^k\subset \mathbb Z^n\times\mathfrak S_n$ et d'une bijection 
$\mathrm{bij} : \mathrm{Ens}_{n-1}^k \times[\![0,n]\!]\to\{x\in \mathrm{Ens}_n^k\times[\![0,n]\!]|
\mathrm{invol}(x)\notin\mathrm{Ens}_n^k\times[\![0,n]\!]\}$, telle que le diagramme \eqref{diag:c:0412} commute 
(d) d'un sous-ensemble $\mathrm{Aff}(\Delta^{n-1},\Delta^n)$ de $\mathbf{Aff}(\mathbb R^{n-1},\mathbb R^n)$ et la 
construction d'applications $f : \mathrm{Ens}_n^k \times[\![0,n]\!]\to \mathrm{Aff}(\Delta^{n-1},\Delta^n)$ et 
$\tilde f : \mathrm{Ens}_{n-1}^k \times[\![0,n]\!] \to \mathrm{Aff}(\Delta^{n-1},\Delta^n)$
telle que le diagramme  \eqref{diag:d:0412} commute. 

\subsubsection{Diagramme commutatif impliquant une involution de $\mathbb Z^n\times\mathfrak S_n\times[\![0,n]\!]$}

On note $\mathbf{Aff}$ la catégorie des espaces affines, dont les morphismes sont les applications affines. 

Pour $n\neq 0$, on note $(e_1^n,\ldots,e_n^n)$ la base canonique de $\mathbb R^n$ et pour $i\in[\![0,n]\!]$, on pose 
$E_i^n:=e^n_n+e^n_{n-1}+\cdots+e^n_{n-i+1}\in \mathbb R^n$ (on a en particulier $E_0^n=0$). 

\begin{defn}\label{def:1709:2311:BIS}
Soit $n,m\geq 0$. Pour $P_0,\ldots,P_n\in\mathbb R^m$, on note $[\underline{P_0,\ldots,P_n}]\in\mathbf{Aff}(\mathbb R^n,\mathbb R^m)$ 
l'unique application affine $\mathbb R^n\to\mathbb R^m$ telle que $E_i^n\mapsto P_i$ pour $i=0,\ldots,n$. 
\end{defn}

\begin{defn}
Pour $i\in[\![0,n]\!]$, on pose $\underline\partial_i^n:=[\underline{E_0^n,\ldots,E_{n-i-1}^n,E_{n-i+1}^n,\ldots,E_n^n}]
\in\mathbf{Aff}(\mathbb R^{n-1},\mathbb R^n)$. 
\end{defn}

Soit $\mathfrak S_n$ le groupe des permutations de $[\![1,n]\!]$. Pour $\sigma\in\mathfrak S_n$, on note 
$\sigma^*$ la permutation de $\mathbb R^n$ donnée par $\sigma^*(t_1,\ldots,t_n):=(t_{\sigma(1)},\ldots,t_{\sigma(n)})$. 
On a alors $\sigma^*(e_i^n)=e_{\sigma^{-1}(i)}^n$, et $(\sigma\tau)^*=\tau^*\circ\sigma^*$ pour $\sigma,\tau\in\mathfrak S_n$.  

\begin{defn}
Pour $(v,\sigma) \in \mathbb Z^n\times\mathfrak S_n$, on définit $\underline c_k(v,\sigma) \in \mathbf{Aff}(\mathbb R^n,\mathbb R^n)$
comme l'application de $\mathbb R^n$ dans lui-même donnée par $x\mapsto (1/k)(v+\sigma^*(x))$.  
\end{defn}

On a alors $\underline c_k(v,\sigma)=[\underline{(1/k)(v+\sigma^*E^n_0),\ldots,(1/k)(v+\sigma^*E^n_n)}]$ pour 
$(v,\sigma) \in \mathbb Z^n\times\mathfrak S_n$. 

\begin{defn}\label{def:2:9:0512}
On note 
$$
\underline f : \mathbb Z^n\times\mathfrak S_n\times[\![0,n]\!]\to \mathbf{Aff}(\mathbb R^{n-1},\mathbb R^n)\quad\operatorname{et}\quad 
\mathrm{sgn}: \mathbb Z^n\times\mathfrak S_n\times[\![0,n]\!]\to\{\pm1\}
$$
les applications données par 
$$
\underline f(v,\sigma,i):= \underline c_k(v,\sigma)\circ\underline\partial_i^n,\quad 
\mathrm{sgn}(v,\sigma,i):=(-1)^i\epsilon(\sigma). 
$$
\end{defn}

\begin{defn}
On note $\mathrm{invol}$ l'application de $\mathbb Z^n\times\mathfrak S_n\times[\![0,n]\!]$ dans lui-même donnée par
$$
\mathrm{invol}(v,\sigma,i):=(v,s_{i,i+1}\circ\sigma,i)\quad\operatorname{pour}\quad 
(v,\sigma,i)\in\mathbb Z^n\times\mathfrak S_n\times[\![1,n-1]\!], 
$$
où $s_{i,i+1}\in\mathfrak S_n$ est la permutation de $i$ et $i+1$, 
$$
\mathrm{invol}(v,\sigma,n):=(v+\sigma^*(e_n^n),c\circ\sigma,0)\quad\operatorname{pour}\quad 
(v,\sigma)\in\mathbb Z^n\times\mathfrak S_n, 
$$
où $c\in\mathfrak S_n$ est le $n$-cycle donné par $c(i):=i+1$ pour $i\neq n$, $c(n)=1$, et 
$$
\mathrm{invol}(v,\sigma,0):=(v-\sigma^*(e_1^n),c^{-1}\circ\sigma,n) \quad\operatorname{pour}\quad 
(v,\sigma)\in\mathbb Z^n\times\mathfrak S_n.  
$$
\end{defn}

\begin{lem}\label{lem:inv:BIS}
(a) $\mathrm{invol}$ est une involution de $\mathbb Z^n\times\mathfrak S_n\times[\![0,n]\!]$. 

(b) Le diagramme suivant commute 
\begin{equation}\label{diag:a:0412}
\xymatrix{
\mathbb Z^n\times\mathfrak S_n \times[\![0,n]\!]
\ar^{\mathrm{invol}}[rr]\ar_{(\underline f,\mathrm{sgn})}[rd]&&
\mathbb Z^n\times\mathfrak S_n \times[\![0,n]\!]
\ar^{(\underline f,-\mathrm{sgn})}[dl]\\
&\mathbf{Aff}(\mathbb R^{n-1},\mathbb R^{n-1})\times\{\pm1\}&}
\end{equation}
\end{lem}

\begin{proof}
(a) Soit $(v,\sigma,i)\in\mathbb Z^n\times\mathfrak S_n\times[\![0,n]\!]$. Si $i\neq 0,n$, alors $s_{i,i+1}^2=id$ implique 
$\mathrm{invol}\circ \mathrm{invol}(v,\sigma,i)=(v,\sigma,i)$. Si $i=n$, alors $\mathrm{invol}\circ\mathrm{invol}(v,\sigma,n)
=\mathrm{invol}(v+\sigma^*(e_n^n),c \circ \sigma,0)=(v+\sigma^*(e_n^n)-(c \circ \sigma)^*(e_1^n),c^{-1} \circ c \circ \sigma,n)
=(v,\sigma,n)$ car $(c\circ\sigma)^*(e_1^n)=\sigma^*\circ c^*(e_1^n)=\sigma^*(e_n^n)$
et si $i=0$, on a $\mathrm{invol}\circ\mathrm{invol}(v,\sigma,0)
=\mathrm{invol}(v-\sigma^*(e_1^n),c^{-1} \circ \sigma,n)
=(v-\sigma^*(e_1^n)+(c^{-1} \circ \sigma)^*(e_n^n),c \circ c^{-1} \circ \sigma,0)
=(v,\sigma,0)$ car $(c^{-1}\circ\sigma)^*(e_n^n)=\sigma^*\circ(c^{-1})^*(e_n^n)=\sigma^*(e_1^n)$. 
On a donc dans tous les cas $\mathrm{invol}\circ\mathrm{invol}(v,\sigma,i)=(v,\sigma,i)$.

(b) Soit $(v,\sigma,i)\in\mathbb Z^n\times\mathfrak S_n\times[\![0,n]\!]$. 

Si $i\neq 0,n$, alors $\mathrm{sgn}\circ\mathrm{invol}(v,\sigma,i)=-\mathrm{sgn}(v,\sigma,i)$ du fait de $\epsilon(s_{i,i+1})=-1$. Si $i=n$, 
alors $\mathrm{sgn}\circ\mathrm{invol}(v,\sigma,n)=\mathrm{sgn}(v+\sigma^*(e_1^n),c^{-1}\circ\sigma,0)=(-1)^n\epsilon(c^{-1}\circ 
\sigma)=-(-1)^0\epsilon(\sigma)=-\mathrm{sgn}(v,\sigma,n)$ du fait de $\epsilon(c)=(-1)^{n-1}$ et $\mathrm{sgn}\circ 
\mathrm{invol}(v,\sigma,0)=\mathrm{sgn}(v-\sigma^*(e_n^n),c\circ\sigma,n)=(-1)^0\epsilon(c\circ\sigma)=-(-1)^n\epsilon(\sigma)
=-\mathrm{sgn}(v,\sigma,0)$ pour la même raison. On a donc dans tous les cas $\mathrm{sgn}\circ\mathrm{invol}(v,\sigma,i)
=-\mathrm{sgn}(v,\sigma,i)$.  

Si $i\neq 0,n$, alors 
\begin{align*}
    & \underline f \circ \mathrm{invol}(v,\sigma,i)=\underline f(v,s_{i,i+1} \circ \sigma,i)
    =[\underline{(v+(s_{i,i+1} \circ \sigma)^*(E_0^n)/k,\ldots,(v+(s_{i,i+1} \circ \sigma)^*(E_n^n)/k}] \\ & 
    \circ [\underline{E_0^n,\ldots,E_{n-i-1}^n,E_{n-i+1}^n,\ldots,E_n^n}]
    =[\underline{(v+(s_{i,i+1} \circ \sigma)^*(E_0^n)/k,\ldots,(v+(s_{i,i+1} \circ \sigma)^*(E_{n-i-1}^n)/k,}\\ & 
    \underline{(v+(s_{i,i+1} \circ \sigma)^*(E_{n-i+1}^n)/k,\ldots,(v+(s_{i,i+1} \circ \sigma)^*(E_n^n)/k}]
    \\ & =[\underline{(v+\sigma^*(E_0^n)/k,\ldots,(v+\sigma^*(E_{n-i-1}^n)/k,(v+\sigma^*(E_{n-i+1}^n)/k,\ldots,(v+\sigma^*(E_n^n)/k}] 
    =\underline f(v,\sigma,i)
\end{align*}
 en utilisant $(\sigma\cdot\tau)^*=\tau^* \circ \sigma^*$ et $s_{i,i+1}^*(E_j^n)=E_j^n$ pour $j\in[\![0,n]\!]$ et $j\neq n-i$. 

Si $i=n$, alors 
\begin{align*} 
    & \underline f \circ \mathrm{invol}(v,\sigma,n)=\underline f(v+\sigma^*(e_n^n),c \circ \sigma,0)
    \\ & =[\underline{(v+\sigma^*(e_n^n)+(c \circ \sigma)^*(E_0^n))/k,\ldots,(v+\sigma^*(e_n^n)+(c \circ \sigma)^*(E_n^n))/k}] 
    \circ [\underline{E_0^n,\ldots,E_{n-1}^n}]
    \\ & =[\underline{(v+\sigma^*(e_n^n)+(c \circ \sigma)^*(E_0^n))/k,\ldots,(v+\sigma^*(e_n^n)+(c \circ \sigma)^*(E_{n-1}^n))/k}]
    \\ & =[\underline{(v+\sigma^*(e_n^n+c^*(E_0^n)))/k,\ldots,(v+\sigma^*(e_n^n+c^*(E_{n-1}^n)))/k}]
    \\ & =[\underline{(v+\sigma^*(E_1^n))/k,\ldots,(v+\sigma^*(E_{n}^n))/k}]
    \\ & =[\underline{(v+\sigma^*(E_0^n))/k,\ldots,(v+\sigma^*(E_n^n))/k}] \circ [\underline{E_1^n,\ldots,E_{n}^n}]=\underline f(v,\sigma,n)
\end{align*}
du fait de $c^*(E_i^n)+e_n^n=E_{i+1}^n$ pour $i\in[\![0,n-1]\!]$. 

Si $i=0$, alors 
\begin{align*}
    & \underline f \circ \mathrm{invol}(v,\sigma,0)=\underline f(v-\sigma^*(e_1^n),c^{-1} \circ \sigma,n)
    \\ & =[\underline{(v-\sigma^*(e_1^n)+(c^{-1} \circ \sigma)^*(E_0^n))/k,\ldots,(v-\sigma^*(e_1^n)+(c^{-1} \circ \sigma)^*(E_n^n))/k}] 
    \circ [\underline{E_1^n,\ldots,E_{n}^n}]
    \\ & =[\underline{(v-\sigma^*(e_1^n)+(c^{-1} \circ \sigma)^*(E_1^n))/k,\ldots,(v-\sigma^*(e_1^n)+(c^{-1} \circ \sigma)^*(E_{n}^n))/k}]
    \\ & =[\underline{(v+\sigma^*(-e_1^n+(c^{-1})^*(E_1^n)))/k,\ldots,(v+\sigma^*(-e_1^n+(c^{-1})^*(E_{n}^n)))/k}]
    \\ & =[\underline{(v+\sigma^*(E_0^n))/k,\ldots,(v+\sigma^*(E_{n-1}^n))/k}]
    \\ & =[\underline{(v+\sigma^*(E_0^n))/k,\ldots,(v+\sigma^*(E_n^n))/k}] \circ [\underline{E_0^n,\ldots,E_{n-1}^n}]=\underline f(v,\sigma,0)
\end{align*}
du fait de $(c^{-1})^*(E_i^n)-e_1^n=E_{i-1}^n$ pour $i \in [\![1,n]\!]$. 
On a donc dans tous les cas $f\circ\mathrm{invol}(v,\sigma,i)=f(v,\sigma,i)$. 
\end{proof}

\begin{rem}
On vérifie directement que $\mathrm{invol}$ est sans point fixe; cela résulte aussi de $\mathrm{sgn}\circ \mathrm{invol}=-\mathrm{sgn}$
et du fait que $\{\pm1\}$ n'a pas de point fixe sous le changement de signe. 
\end{rem}

\subsubsection{Constructions et résultats relatifs aux permutations}

\begin{lem}\label{lem:du:08:11:BIS}
Soit $\tau\in\mathfrak S_{n-1}$. Pour $i\in[\![1,n-1]\!]$, soit $\mathbf{inv}(\tau,i):=\{j\in [\![1,n-1]\!]|(j-i)(\tau(j)-\tau(i))<0\}$. 
Alors $|\mathbf{inv}(\tau,i)|\equiv \tau(i)-i$ (mod 2). 
\end{lem}

\begin{proof}
Montrons l'énoncé par récurrence sur $i$. Si $i=1$, $\mathbf{inv}(\tau,1)=\{j\in [\![1,n-1]\!]|\tau(j)<\tau(1)\}
=\tau^{-1}([\![1,\tau(1)-1]\!])$, donc $|\mathbf{inv}(\tau,1)|=|\tau^{-1}([\![1,\tau(1)-1]\!])|
=|[\![1,\tau(1)-1]\!]|=\tau(1)-1$ où la deuxième égalité suit de la bijectivité de $\tau$, ce qui implique l'énoncé pour 
$i=1$. 

Soit $1\leq i<n-1$ : supposons l'énoncé vrai pour $i$ et montrons-le pour $i+1$. Pour cela, on note 
$\mathbf A_0:=\{j|j<i$ et $\tau(j)>\tau(i)\}$, 
$\mathbf A_1:=\{j|j<i$ et $\tau(j)>\tau(i+1)\}$, $\mathbf B_0:=\{j|j>i+1$ et $\tau(j)<\tau(i)\}$, 
$\mathbf B_1:=\{j|j>i+1$ et $\tau(j)<\tau(i+1)\}$ ; on observe que 
$\mathbf A_\alpha\cap \mathbf B_\beta=\emptyset$ pour tous $\alpha,\beta\in\{0,1\}$. 

Deux cas se présentent : 

$\bullet$ on a $\tau(i)<\tau(i+1)$. On a alors $i+1\notin \mathbf{inv}(\tau,i)$ et $i\notin \mathbf{inv}(\tau,i+1)$, ce qui implique 
$\mathbf{inv}(\tau,i)=\mathbf A_0\cup \mathbf B_0$ et $\mathbf{inv}(\tau,i+1)=
\mathbf A_1\cup \mathbf B_1$. En notant $\triangle$ l'opération de différence symétrique, on obtient 
\begin{equation}\label{toto:0811:BIS}
 \mathbf{inv}(\tau,i)\triangle\mathbf{inv}(\tau,i+1)=(\mathbf A_0\triangle\mathbf A_1)\cup(\mathbf B_0\triangle\mathbf B_1),    
\end{equation}
compte tenu de $(X_0\cup Y_0)\triangle(X_1\cup Y_1)=(X_0\triangle X_1)\cup (Y_0\triangle Y_1)$
pour tous ensembles $X_0,X_1,Y_0,Y_1$ tels que $X_\alpha\cap Y_\beta=\emptyset$. 
Du fait que $\tau(i)<\tau(i+1)$, on a  $\mathbf A_1\subset \mathbf A_0$ et $\mathbf B_1\supset \mathbf B_0$, ce qui implique 
\begin{equation}\label{tutu:0811:BIS}
    (\mathbf A_0\triangle\mathbf A_1)\cup(\mathbf B_0\triangle\mathbf B_1)
=(\mathbf A_0-\mathbf A_1)\cup(\mathbf B_1-\mathbf B_0).
\end{equation}
Alors 
$$
\mathbf A_0-\mathbf A_1=\tau^{-1}([\![\tau(i)+1,\tau(i+1)]\!])\cap\{j|j<i\},
\quad 
\mathbf B_1-\mathbf B_0=\tau^{-1}([\![\tau(i),\tau(i+1)-1]\!])\cap\{j|j>i+1\}.  
$$
Comme $\tau(i+1)\notin\tau(\{j|j<i\})$ et $\tau(i)\notin\tau(\{j|j>i+1\})$ on en déduit 
$$
\mathbf A_0-\mathbf A_1=\tau^{-1}([\![\tau(i)+1,\tau(i+1)-1]\!])\cap\{j|j<i\},\quad 
\mathbf B_1-\mathbf B_0=\tau^{-1}([\![\tau(i)+1,\tau(i+1)-1]\!])\cap\{j|j>i+1\}, 
$$
donc 
$(\mathbf A_0-\mathbf A_1)\cup(\mathbf B_1-\mathbf B_0)=\tau^{-1}([\![\tau(i)+1,\tau(i+1)-1]\!])\cap\{j|j\neq i,i+1\}$. 
Compte tenu de $\{i,i+1\}\cap \tau^{-1}([\![\tau(i)+1,\tau(i+1)-1]\!])=\emptyset$, on en déduit 
\begin{equation}\label{titi:0811:BIS}
    (\mathbf A_0-\mathbf A_1)\cup(\mathbf B_1-\mathbf B_0)=\tau^{-1}([\![\tau(i)+1,\tau(i+1)-1]\!]). 
\end{equation}
Alors 
\begin{align}\label{NUMERO:2811:BIS}
& \nonumber |\mathbf{inv}(\tau,i+1)|-|\mathbf{inv}(\tau,i)|\equiv|\mathbf{inv}(\tau,i)\triangle\mathbf{inv}(\tau,i+1)|
=|(\mathbf A_0-\mathbf A_1)\cup(\mathbf B_1-\mathbf B_0)|
\\ & =|\tau^{-1}([\![\tau(i)+1,\tau(i+1)-1]\!])|=|[\![\tau(i)+1,\tau(i+1)-1]\!]|=\tau(i+1)-\tau(i)-1
\end{align}
mod 2, où la première égalité suit de 
\begin{equation}\label{tyty:0811:BIS}
    |A\triangle B|\equiv |B|-|A|\text{ mod }2\text{ pour }A,B\text{ ensembles finis,}
\end{equation}
la deuxième égalité suit de la combinaison de \eqref{toto:0811:BIS} et \eqref{tutu:0811:BIS}, la troisième égalité suit de \eqref{titi:0811:BIS}, 
la quatrième égalité suit de la bijectivité de $\tau$ ; 

$\bullet$ on a $\tau(i+1)<\tau(i)$. 
On a alors $\mathbf A_1\supset \mathbf A_0$ et $\mathbf B_1\subset \mathbf B_0$, ce qui implique 
\begin{equation}\label{tutu:0811:prime:BIS}
    (\mathbf A_0\triangle\mathbf A_1)\cup(\mathbf B_0\triangle\mathbf B_1)
=(\mathbf A_1-\mathbf A_0)\cup(\mathbf B_0-\mathbf B_1). 
\end{equation}
Alors
$$
\mathbf A_1-\mathbf A_0=\tau^{-1}([\![\tau(i+1)+1,\tau(i)]\!])\cap\{j|j<i\},\quad 
\mathbf B_0-\mathbf B_1=\tau^{-1}([\![\tau(i+1),\tau(i)-1]\!])\cap\{j|j>i+1\}.
$$
Or $\tau(i)\notin\tau(\{j|j<i\})$, $\tau(i+1)\notin\tau(\{j|j>i+1\})$, donc 
$$
\mathbf A_1-\mathbf A_0=\tau^{-1}([\![\tau(i+1)+1,\tau(i)-1]\!])\cap\{j|j<i\},\quad 
\mathbf B_0-\mathbf B_1=\tau^{-1}([\![\tau(i+1)+1,\tau(i)-1]\!])\cap\{j|j>i+1\}.
$$
Donc $(\mathbf A_1-\mathbf A_0)\cup(\mathbf B_0-\mathbf B_1)=\tau^{-1}([\![\tau(i+1)+1,\tau(i)-1]\!])\cap\{j|j\neq i,i+1\}$. 
Compte tenu de $\{i,i+1\}\cap \tau^{-1}([\![\tau(i+1)+1,\tau(i)-1]\!])=\emptyset$, on en déduit 
\begin{equation}\label{titi:0811:prime:BIS}
    (\mathbf A_1-\mathbf A_0)\cup(\mathbf B_0-\mathbf B_1)=\tau^{-1}([\![\tau(i+1)+1,\tau(i)-1]\!]). 
\end{equation}
De plus, comme $\tau(i+1)<\tau(i)$, on a 
$\mathbf{inv}(\tau,i)=\mathbf A_0\cup \mathbf B_0\cup \{i+1\}$ et 
$\mathbf{inv}(\tau,i+1)=\mathbf A_1\cup \mathbf B_1\cup \{i\}$. On a $i\notin\mathbf{inv}(\tau,i)$ et 
$i+1\notin\mathbf{inv}(\tau,i+1)$, donc 
\begin{equation}\label{neweq:0811:BIS}
\mathbf{inv}(\tau,i)\triangle\mathbf{inv}(\tau,i+1)=((\mathbf A_0\cup \mathbf B_0)\triangle(\mathbf A_1\cup \mathbf B_1))
\cup\{i,i+1\}. 
\end{equation}
Alors 
\begin{align}\label{NUMERO':2811:BIS}
&\nonumber |\mathbf{inv}(\tau,i)|-|\mathbf{inv}(\tau,i+1)|\equiv |\mathbf{inv}(\tau,i)\triangle\mathbf{inv}(\tau,i+1)|
=|(\mathbf A_0\cup \mathbf B_0)\triangle(\mathbf A_1\cup \mathbf B_1)|-2
\\ & \equiv |(\mathbf A_0\cup \mathbf B_0)\triangle(\mathbf A_1\cup \mathbf B_1)|
=|\tau^{-1}([\![\tau(i+1)+1,\tau(i)-1]\!])|=|[\![\tau(i+1)+1,\tau(i)-1]\!]|=\tau(i)-\tau(i+1)-1. 
\end{align}
où la première égalité suit de  \eqref{tyty:0811:BIS} et $x\equiv -x$ mod 2, la deuxième égalité suit de \eqref{neweq:0811:BIS}
combiné à $\{i,i+1\}\cap ((\mathbf A_0\cup \mathbf B_0)\triangle(\mathbf A_1\cup \mathbf B_1))\subset 
\{i,i+1\}\cap \{j|j\neq i,i+1\}=\emptyset$, la quatrième égalité suit de la combinaison de \eqref{toto:0811:BIS} et 
\eqref{tutu:0811:prime:BIS}, la cinquième égalité suit de la bijectivité de $\tau$. 

On déduit des égalités \eqref{NUMERO:2811:BIS} dans le premier cas et \eqref{NUMERO':2811:BIS} dans le second l'égalité  $|\mathbf{inv}(i+1)|-\tau(i+1)+i+1\equiv |\mathbf{inv}(i)|-\tau(i)+i$ 
mod 2. On a $|\mathbf{inv}(\tau,i)|-\tau(i)+i\equiv 0$ mod 2 d'après l'hypothèse de récurrence, ce qui implique 
$|\mathbf{inv}(\tau,i+1)|-\tau(i+1)+i+1\equiv 0$ mod 2. 
\end{proof}

\begin{defn} Soit $i\in[\![1,n-1]\!]$.

(a)  On note $p_i  : [\![1,n]\!]\to[\![1,n-1]\!]$ l'application donnée par 
$p_i(x)=x$ si $x \leq i$ et $p_i(x)=x-1$ si $x \geq i+1$. 

(b) On note $\mathrm{st}_i : [\![1,n-1]\!]\to[\![1,n]\!]$ l'application donnée par $x\mapsto x$ si $x\leq i$ et 
$x\mapsto x+1$ si $x\geq i+1$.  
\end{defn}

\begin{lemdef}
Soit $\tau\in\mathfrak S_{n-1}$ et $i\in[\![0,n]\!]$. Il existe un unique élément $\tau^{(i)}\in \mathfrak S_n$ satisfaisant 
les conditions suivantes : 

(a) $p_{\tau(i)}\circ\tau^{(i)}=\tau\circ p_i$
et $\tau^{(i)}(i)=\tau(i)$, $\tau^{(i)}(i+1)=\tau(i)+1$ si $i\neq 0,n$ ;  

(b) $\tau^{(0)}(1)=1$ et $\tau^{(0)}(x)=\tau(x-1)+1$ pour tout $x\in [\![2,n]\!]$ si $i=0$ ; 

(c) $\tau^{(n)}(n)=n$ et $\tau^{(n)}(x)=\tau(x)$ pour tout $x\in [\![1,n-1]\!]$ si $i=n$. 
\end{lemdef}

\begin{proof}
(a) Pour $j\in [\![0,n]\!]$, $\mathrm{st}_j\circ p_j$ est l'application de $[\![1,n]\!]$ dans lui-même
telle que $x\mapsto x$ pour $x\neq j+1$ et $j+1\mapsto j$. Si $\tau^{(i)}$ est une application de 
$[\![1,n]\!]$ dans lui-même satisfaisant les conditions dites, on a alors  
$\mathrm{st}_{\tau(i)}\circ\circ p_{\tau(i)}\circ\tau^{(i)}=\mathrm{st}_{\tau(i)}\circ\tau\circ p_i$
ce qui implique $\tau^{(i)}(x)=\mathrm{st}_{\tau(i)}\circ\tau\circ p_i(x)$
pour tout $x\neq i,i+1$ ainsi que $\tau^{(i)}(i)=\tau(i)$, $\tau^{(i)}(i+1)=\tau(i)+1$. 
Les conditions dites déterminent donc uniquement $\tau^{(i)}$ comme application de 
$[\![1,n]\!]$ dans lui-même. On vérifie alors que $\tau^{(i)}\in\mathfrak S_n$.  
(b,c) Pour $i=0,n$, l'application $\tau^{(i)}$ est la juxtaposition de deux permutations, donc est une permutation. 
\end{proof}

\begin{lem}\label{lem:124:1611:BIS}
Si $\tau\in\mathfrak S_{n-1}$ et $i\in[\![1,n-1]\!]$, alors $\epsilon(\tau^{(i)})=\epsilon(\tau)(-1)^{\tau(i)-i}$. 
\end{lem}

\begin{proof}
Si $p\geq 1$ et $\sigma\in\mathfrak S_p$, notons $\mathbf{inv}(\sigma):=\{(a,b)\in[\![1,p]\!]^2|a<b$ et $\sigma(a)>\sigma(b)\}$. On a 
alors 
\begin{equation}\label{def:eps:0811:BIS}
    \epsilon(\sigma)=(-1)^{|\mathbf{inv}(\sigma)|}. 
\end{equation}

Si $\tau\in\mathfrak S_{n-1}$ et $i\in[\![1,n-1]\!]$, on a une partition
\begin{equation}\label{dec:1:0811:BIS}
\mathbf{inv}(\tau)=A\sqcup B,     
\end{equation}
avec $A:=\{(a,b)\in\mathbf{inv}(\tau)|a\neq i$ et $b\neq i\}$ et $B:=\{(a,b)\in\mathbf{inv}(\tau)|a=i$ ou $b=i\}$. 
On a de même une partition 
\begin{equation}\label{dec:2:0811:BIS}
\mathbf{inv}(\tau^{(i)})=A'\sqcup B'\sqcup B'', 
\end{equation}
avec 
$$
A':=\{(a,b)\in\mathbf{inv}(\tau^{(i)})|a\notin\{i,i+1\}\text{ et }b\notin\{i,i+1\}\}, 
$$
$$
B':=\{(a,b)\in\mathbf{inv}(\tau^{(i)})|(a=i\text{ et }b\notin\{i,i+1\})\text{ ou }(b=i\text{ et }a\notin\{i,i+1\})\}
$$
$$
B'':=\{(a,b)\in\mathbf{inv}(\tau^{(i)})|(a=i+1\text{ et }b\notin\{i,i+1\})\text{ ou }(b=i+1\text{ et }a\notin\{i,i+1\})\}
$$
ceci du fait que $(i,i+1)\notin\mathbf{inv}(\tau^{(i)})$ car $\tau^{(i)}(i+1)=1+\tau^{(i)}(i)$.

On a encore des bijections  
\begin{equation}\label{isos:1:0811:BIS}
A\stackrel{\sim}{\to} A', \quad B\stackrel{\sim}{\to} B',\quad B\stackrel{\sim}{\to} B''
\end{equation}
induites respectivement par $(a,b)\mapsto (\mathrm{st}_i(a),\mathrm{st}_i(b))$ (application $A\to A'$), 
$(a,i)\mapsto (a,i)$ et $(i,b)\mapsto (i,b+1)$ (application $B\to B'$),
$(a,i)\mapsto (a,i+1)$ et $(i,b)\mapsto (i+1,b+1)$ (application $B\to B''$). Enfin on a une bijection 
\begin{equation}\label{isos:2:0811:BIS}
\mathbf{inv}(\tau,i)\stackrel{\sim}{\to} B
\end{equation}
donnée par $a\mapsto (a,i)$ si $a<i$ et $a\mapsto (i,a)$ si $i<a$. 

On a alors 
\begin{equation}\label{deja:fait:0811:BIS}
|\mathbf{inv}(\tau^{(i)})|=|A'|\sqcup |B'|\sqcup |B''|=(|A|\sqcup |B|)\sqcup |B|
=|\mathbf{inv}(\tau)|+|\mathbf{inv}(\tau,i)|
\equiv|\mathbf{inv}(\tau)|+\tau(i)-i 
\end{equation}
mod 2, où la première (resp. deuxième, troisième, quatrième) égalité suit de \eqref{dec:2:0811:BIS} (resp. \eqref{isos:1:0811:BIS}, 
\eqref{isos:2:0811:BIS}, lemme \ref{lem:du:08:11:BIS}). On a alors
$$
\epsilon(\tau^{(i)})=(-1)^{|\mathbf{inv}(\tau^{(i)})|}=(-1)^{|\mathbf{inv}(\tau)|+\tau(i)-i}
=(-1)^{|\mathbf{inv}(\tau)|}(-1)^{\tau(i)-i}=\epsilon(\tau)(-1)^{\tau(i)-i}.
$$
où la première (resp. deuxième, dernière) égalité suit de \eqref{def:eps:0811:BIS} (resp. \eqref{deja:fait:0811:BIS}, \eqref{def:eps:0811:BIS}). 
\end{proof}

\subsubsection{Diagramme commutatif impliquant les applications $(\underline f,\mathrm{sgn})$, $(\underline{\tilde f},\widetilde{\mathrm{sgn}})$ 
et $\underline{\mathrm{bij}}$}

\begin{defn}\label{def:2:17:0512}
On note 
$$
\underline{\tilde f} : \mathbb Z^{n-1}\times\mathfrak S_{n-1}\times[\![0,n]\!]\to \mathbf{Aff}(\mathbb R^{n-1},\mathbb R^n)\quad\operatorname{et}\quad 
\widetilde{\mathrm{sgn}}: \mathbb Z^{n-1}\times\mathfrak S_{n-1}\times[\![0,n]\!]\to\{\pm1\}
$$
les applications données par 
$$
\underline{\tilde f}(\tilde v,\tilde \sigma,i):=\underline\partial_i^n\circ \underline c_k(\tilde v,\tilde\sigma),\quad 
\widetilde{\mathrm{sgn}}(\tilde v,\tilde\sigma,i):=(-1)^i\epsilon(\tilde\sigma). 
$$
\end{defn}

\begin{defn}\label{def:transfo:w:BIS}
On note $w\mapsto w^{(0)}$ et $w\mapsto w^{(n)}$ les applications $\mathbb Z^{n-1}\to\mathbb Z^n$ données par 
$w^{(0)}:=(0,w)$ et $w^{(n)}:=(w,k-1)$. 
\end{defn}

\begin{defn}
On note $\underline{\mathrm{bij}} : \mathbb Z^{n-1}\times\mathfrak S_{n-1}\times[\![0,n]\!]\to 
\mathbb Z^n\times\mathfrak S_n\times[\![0,n]\!]$ l'application donnée par 
$(w,\tau,i)\mapsto (w \circ p_i,\tau^{(i)},\tau(i))$ 
si $i\neq 0,n$, par 
$(w,\tau,0)\mapsto(w^{(0)},\tau^{(0)},0)$, et par $(w,\tau,n)\mapsto (w^{(n)},\tau^{(n)},n)$. 
\end{defn}

\begin{lem}\label{lem:composition:BIS}
Le diagramme suivant commute
\begin{equation}\label{diag:b:0412}
\xymatrix{
\mathbb Z^{n-1}\times\mathfrak S_{n-1} \times[\![0,n]\!]
\ar^{\underline{\mathrm{bij}}}[rr]\ar_{(\underline{\tilde f},\widetilde{\mathrm{sgn}})}[rd]&&
\mathbb Z^n\times\mathfrak S_n \times[\![0,n]\!]
\ar^{(\underline f,\mathrm{sgn})}[ld]\\&\mathbf{Aff}(\mathbb R^{n-1},\mathbb R^{n-1})\times\{\pm1\}&}
\end{equation}
\end{lem}

\begin{proof}
Soit $(w,\tau,i)\in \mathbb Z^{n-1}\times\mathfrak S_{n-1}\times[\![0,n]\!]$. Supposons $i\neq 0,n$. 
Alors
$$
\mathrm{sgn}\circ\underline{\mathrm{bij}}(w,\tau,i)
=\mathrm{sgn}(w\circ p_i,\tau^{(i)},\tau(i))=\epsilon(\tau^{(i)})(-1)^{\tau(i)}
=\epsilon(\tau)(-1)^{i}=\widetilde{\mathrm{sgn}}(w,\tau,i) 
$$
d'après le lemme \ref{lem:124:1611:BIS}. De plus 
\begin{align*}
&\underline f\circ\underline{\mathrm{bij}}(w,\tau,i)
=\underline f(w\circ p_i,\tau^{(i)},\tau(i))
=\underline c_k(w\circ p_i,\tau^{(i)})\circ\underline\partial_{\tau(i)}^n
\\ & 
=[(t_1,\ldots,t_{n-1})\mapsto (w\circ p_i+(\tau^{(i)})^*(t_1,\ldots,t_{\tau(i)},t_{\tau(i)},\ldots,t_{n-1}))/k]
\\ & 
=[(t_1,\ldots,t_{n-1})\mapsto (w\circ p_i+(\tau^{(i)})^*(t_{p_{\tau(i)}(1)},\ldots,t_{p_{\tau(i)}(n)}))/k]
\\ & 
=[(t_1,\ldots,t_{n-1})\mapsto (w\circ p_i+(t_{p_{\tau(i)} \circ \tau^{(i)}(1)},\ldots,t_{p_{\tau(i)} \circ \tau^{(i)}(n)}))/k]
\\ & =[(t_1,\ldots,t_{n-1})\mapsto (w\circ p_i+(t_{\tau\circ p_i(1)},\ldots,t_{\tau\circ p_i(n)}))/k]
\\ & =\underline\partial_i^n\circ [(t_1,\ldots,t_{n-1})\mapsto (w+(t_{\tau(1)},\ldots,t_{\tau(n-1)}))/k]
\\ & =\underline\partial_i^n\circ [(t_1,\ldots,t_{n-1})\mapsto (w+\tau^*(t_1,\ldots,t_{n-1}))/k]
=\underline\partial_i^n\circ \underline c_k(w,\tau)=\underline{\tilde f}(w,\tau,i)
\end{align*}
où la sixième égalité provient de $\tau \circ p_{i} =p_{\tau(i)}\circ \tau^{(i)}$ (égalité d'applications $[\![1,n]\!]\to
[\![1,n-1]\!]$) et la septième égalité suit de 
$\underline\partial_i^n(x)=x\circ p_i$ pour $x\in\mathbb R^{n-1}=\mathrm{Appl}([\![1,n-1]\!],\mathbb R)$. 

Si $i=0$, alors 
\begin{align*}
    & \underline f\circ\underline{\mathrm{bij}}(w,\tau,0)=\underline f(w^{(0)},\tau^{(0)},0)
    =\underline c_k(w^{(0)},\tau^{(0)}) \circ \underline\partial^n_0
    \\ & =[(t_1,\ldots,t_{n-1})\mapsto (w^{(0)}+(\tau^{(0)})^*(0,t_1,\ldots,t_{n-1}))/k]
    \\ & =[(t_1,\ldots,t_{n-1})\mapsto ((0,w)+(0,\tau^*(t_1,\ldots,t_{n-1})))/k]
    \\ & =[(t_1,\ldots,t_{n-1})\mapsto (0,w+\tau^*(t_1,\ldots,t_{n-1}))/k]
    =\underline\partial^n_0\circ \underline c_k(w,\tau)=\underline{\tilde f}(w,\tau,0). 
\end{align*}

Si $i=n$, alors 
\begin{align*}
    & \underline f\circ\underline{\mathrm{bij}}(w,\tau,n)=\underline f(w^{(n)},\tau^{(n)},n)
    =\underline c_k(w^{(n)},\tau^{(n)}) \circ \underline\partial^n_n
    \\ & =[(t_1,\ldots,t_{n-1})\mapsto (w^{(n)}+(\tau^{(n)})^*(t_1,\ldots,t_{n-1},1))/k]
    \\ & =[(t_1,\ldots,t_{n-1})\mapsto ((w,k-1)+(\tau^*(t_1,\ldots,t_{n-1}),1))/k]
    \\ & =[(t_1,\ldots,t_{n-1})\mapsto (w+\tau^*(t_1,\ldots,t_{n-1}),k)/k]
    \\ & =[(t_1,\ldots,t_{n-1})\mapsto ((w+\tau^*(t_1,\ldots,t_{n-1}))/k,1)]
    \\ & = \underline\partial^n_n\circ [(t_1,\ldots,t_{n-1})\mapsto (w+\tau^*(t_1,\ldots,t_{n-1}))/k]
    =\underline\partial^n_n\circ \underline c_k(w,\tau)=\underline{\tilde f}(w,\tau,n). 
\end{align*}
\end{proof}

\subsubsection{Battages et transpositions}

\begin{defn}
Si $n\geq 1$ et $n_1,\ldots,n_k$ sont des entiers positifs ou nuls avec $n_1+\cdots+n_k=n$, on note $\mathfrak S_{n_1,\ldots,n_k}$ 
l'ensemble des éléments $\sigma\in\mathfrak S_n$ tel que pour tout $i=1,\ldots,k$, la restriction de $\sigma$ à 
$n_1+\cdots+n_{i-1}+[\![1,n_i]\!]$ est croissante. 
\end{defn}

Si $n_i=0$, l'ensemble $[\![1,n_i]\!]$ est vide, la condition relative à $i$ est alors automatiquement satisfaite. 

\begin{lem}\label{lem:toto:0911:BIS}
Soit $n_1,\ldots,n_k$ des entiers $\geq 1$ et $\sigma \in \mathfrak S_{n_1,\ldots,n_k}$ et $i\in[\![1,n-1]\!]$ avec 
$n:=n_1+\cdots+n_k$. Alors $s_{i,i+1} \circ \sigma \notin \mathfrak S_{n_1,\ldots,n_k}$ si et seulement si 
$\sigma^{-1}(i) \notin \{n_1,n_1+n_2,\ldots,n_1+\cdots+n_k\}$ et $\sigma^{-1}(i+1)=\sigma^{-1}(i)+1$. 
\end{lem}

\begin{proof} Soit $\sigma \in \mathfrak S_{n_1,\ldots,n_k}$ et $i\in[\![1,n-1]\!]$. 
Soit $j:=\sigma^{-1}(i)$, $j':=\sigma^{-1}(i+1)$. Pour $\alpha\in[\![1,k]\!]$, notons $I_\alpha:=
n_1+\cdots+n_{\alpha-1}+[\![1,n_\alpha]\!]$. Alors on a une partition $[\![1,n]\!]=\sqcup_{\alpha=1}^k I_\alpha$. 
Soit $\alpha,\alpha'\in[\![1,k]\!]$ les indices tels que 
$j\in I_\alpha$, $j'\in I_{\alpha'}$. Montrer l'équivalence annoncée, on montre d'abord
l'équivalence $(\alpha=\alpha')\iff (s_{i,i+1}\circ \sigma\notin\mathfrak S_{n_1,\ldots,n_k})$, 
puis l'équivalence $(\alpha=\alpha')\iff (\sigma^{-1}(i) \notin \{n_1,n_1+n_2,\ldots,n_1+\cdots+n_k\}$ 
et $\sigma^{-1}(i+1)=\sigma^{-1}(i)+1)$. 

{\it Première étape : équivalence $(\alpha=\alpha')\iff (s_{i,i+1}\circ \sigma\notin\mathfrak S_{n_1,\ldots,n_k})$.}
(a) Supposons $\alpha=\alpha'$. Alors on a $j,j'\in I_\alpha$. La restriction $\sigma_{|I_\alpha}$ est strictement croissante,
et $\sigma(j)=i$, $\sigma(j')=i+1$, donc $j<j'$. De plus, si on avait $j'>j+1$, alors $j+1\in I_\alpha$ et 
$\sigma(j)<\sigma(j+1)<\sigma(j')$ donc $i<\sigma(j+1)<i+1$ ce qui est impossible, $\sigma(j+1)$ étant entier. 
Donc $j'=j+1$. Comme $j+1\in I_\alpha$, on a nécessairement $j\neq n_1+\cdots+n_\alpha$. De plus, 
la restriction de $s_{i,i+1}\circ \sigma$ à $I_\alpha$ est telle que $j\mapsto i+1$ et $j+1\mapsto i$ ; 
cette restriction n'est donc pas strictement croissante, donc $s_{i,i+1}\circ \sigma\notin\mathfrak S_{n_1,\ldots,n_k}$. 

Supposons $\alpha\neq\alpha'$. Comme $\sigma(I_\alpha)\notni i+1$, la restriction de $s_{i,i+1}\circ \sigma$ à $I_\alpha$ est 
égale à $a\circ \sigma_{|I_\alpha}$, où $a : \sigma(I_\alpha)\to [\![1,n]\!]$ est donnée par $x\mapsto x$ si $x\neq i$ et $i\mapsto i+1$. 
L'application $a$ est croissante, donc il est de même de  
$(s_{i,i+1}\circ \sigma)_{|I_\alpha}$. De même, $\sigma(I_{\alpha'})\notni i$, 
donc la restriction de $s_{i,i+1}\circ \sigma$ à $I_{\alpha'}$ est égale à $a'\circ \sigma_{|I_\alpha}$, où 
$a' : \sigma(I_{\alpha'})\to [\![1,n]\!]$ est donnée par $x\mapsto x$ si $x\neq i+1$ et $i+1\mapsto i$. 
L'application $a'$ est croissante, donc il est de même de  $(s_{i,i+1}\circ \sigma)_{|I_{\alpha'}}$. Enfin pour tout 
$\beta\neq\alpha,\alpha'$, on a $(s_{i,i+1}\circ \sigma)_{|I_\beta}=\sigma_{|I_\beta}$ donc $(s_{i,i+1}\circ \sigma)_{|I_\beta}$ est 
croissante. Donc $s_{i,i+1}\circ \sigma\in\mathfrak S_{n_1,\ldots,n_k}$. 

On a donc équivalence entre $s_{i,i+1}\circ \sigma\notin\mathfrak S_{n_1,\ldots,n_k}$ et $\alpha=\alpha'$. 

{\it Seconde étape : $(\alpha=\alpha')\iff (\sigma^{-1}(i) \notin \{n_1,n_1+n_2,\ldots,n_1+\cdots+n_k\}$ 
et $\sigma^{-1}(i+1)=\sigma^{-1}(i)+1)$.}
On a vu que si $\alpha=\alpha'$, alors $j'=j+1$ et $j\neq n_1+\cdots+n_\alpha$. Comme $j\in I_\alpha$, on a aussi 
$j\neq n_1+\cdots+n_\beta$ pour tout $\beta\neq\alpha$, donc $j\notin\{n_1,\ldots,n_1+\cdots+n_k\}$. 

Inversement, si $j'=j+1$ et $j\notin\{n_1,\ldots,n_1+\cdots+n_k\}$, on a $j\neq n_1+\cdots+n_\alpha$ donc 
$j\in n_1+\cdots+n_{\alpha-1}+[\![1,n_\alpha-1]\!]$ donc $j'=j+1\in n_1+\cdots+n_{\alpha-1}+[\![1,n_\alpha]\!]=I_\alpha$, 
donc $\alpha=\alpha'$. 
\end{proof}

\subsubsection{La bijection $\mathrm{bij}$}

\begin{defn}\label{def:ens:n:k:1210:BIS}
(a) On note $[\![0,k-1]\!]^n_{\leq}:=\{(v_1,\ldots,v_n)\in \mathbb Z^n|0\leq v_1\leq\ldots\leq v_n\leq k-1\}$.

(b) On note $\mathrm{Ens}_n^k$ l'ensemble des couples $(v,\sigma)\in [\![0,k-1]\!]^n_{\leq}\times\mathfrak S_n$, tels que
$\sigma\in\mathfrak S_{|v^{-1}(0)|,\ldots,|v^{-1}(k-1)|}$. 
\end{defn} 
Si $(v,\sigma)\in \mathrm{Ens}_n^k$, on a pour tout $i\in [\![1,n-1]\!]$ l'implication
\begin{equation}\label{TOTO:2311:BIS}
    (v_i=v_{i+1})\implies(\sigma(i)<\sigma(i+1)). 
\end{equation}

\begin{lem}
Soit $(v,\sigma,i) \in \mathrm{Ens}_n^k\times[\![0,n]\!]$. 

(a) Si $i \neq 0,n$, la condition $\mathrm{invol}(v,\sigma,i) \notin \mathrm{Ens}_n^k\times[\![0,n]\!]$ est équivalente à 
la conjonction de 
$\sigma^{-1}(i)\notin\{|v^{-1}(0)|,\ldots,|v^{-1}(0)|+\cdots+|v^{-1}(k-1)|\}$ et $\sigma^{-1}(i+1)=\sigma^{-1}(i)+1$.

(b) Si $i=n$, la condition $\mathrm{invol}(v,\sigma,i) \notin \mathrm{Ens}_n^k\times[\![0,n]\!]$ est équivalente à la conjonction 
$\sigma(n)=n$ et $v(n)=k-1$ ; 

(c) Si $i=0$, la condition $\mathrm{invol}(v,\sigma,i) \notin \mathrm{Ens}_n^k\times[\![0,n]\!]$ est équivalente à la conjonction 
$\sigma(1)=1$ et $v(1)=0$. 
\end{lem}

\begin{proof} (a) On a $\sigma\in\mathfrak S_{|v^{-1}(0)|,\ldots,|v^{-1}(0)|+\cdots+|v^{-1}(k-1)|}$ et 
$\mathrm{invol}(v,\sigma,i)=(v,s_{i,i+1}\circ\sigma,i)$. On a $v\in[\![0,k-1]\!]^n_{\leq}$ donc on a 
équivalence entre $\mathrm{invol}(v,\sigma,i)\notin \mathrm{Ens}_n^k\times[\![0,n]\!]$ et 
$s_{i,i+1}\circ\sigma\notin \mathfrak S_{|v^{-1}(0)|,\ldots,|v^{-1}(0)|+\cdots+|v^{-1}(k-1)|}$. Le résultat est alors conséquence du lemme
\ref{lem:toto:0911:BIS}. 

(b) Rappelons que $\mathrm{invol}(v,\sigma,n)=(v+e^n_{\sigma^{-1}(n)},c\circ\sigma,0)$. 

Notons $n_\alpha:=|v^{-1}(\alpha-1)|$ pour $\alpha=1,\ldots,k$. Alors on a une partition $[\![1,n]\!]=\sqcup_{\alpha=1}^k I_\alpha$, 
avec $I_\alpha=n_1+\cdots+n_{\alpha-1}+[\![1,n_\alpha]\!]$ (avec la convention que $I_\alpha=\emptyset$ si $n_\alpha=0$), et 
$v$ prend la valeur $\alpha-1$ sur $I_\alpha$, pour tout $\alpha$. Rappelons l'identification de $[\![0,k-1]\!]^n_{\leq}$ 
à l'ensemble $\mathrm{Appl}_{\leq}([\![1,n]\!],[\![0,k-1]\!])$. 
On a $\sigma^{-1}(n)\in\{n_1+\cdots+n_\beta|n_\beta\neq 0\}$ donc 
$v+e^n_{\sigma^{-1}(n)}\in \mathrm{Appl}_{\leq}([\![1,n]\!],[\![0,k]\!])$. De plus, $v+e^n_{\sigma^{-1}(n)}$ atteint la valeur 
$k$ si et seulement si $n_k\neq 0$ et $\sigma^{-1}(n)=n_1+\cdots+n_k$. La conjonction de ces conditions est équivalente à celle de $v(n)=k-1$
et $\sigma(n)=n$. Ceci montre l'équivalence entre $v+e^n_{\sigma^{-1}(n)}\notin [\![0,k-1]\!]^n_{\leq}$ et la conjonction des conditions 
$v(n)=k-1$ et $\sigma(n)=n$.

Alors si $v(n)=k-1$ et $\sigma(n)=n$, on a $v+e^n_{\sigma^{-1}(n)}\notin [\![0,k-1]\!]^n_{\leq}$
donc $\mathrm{invol}(v,\sigma,0)\notin\mathrm{Ens}_n^k\times[\![0,n]\!]$. 

Si $v(n)<k-1$ ou $\sigma(n)\neq n$, montrons que $\mathrm{invol}(v,\sigma,0)\in\mathrm{Ens}_n^k\times[\![0,n]\!]$.
D'une part, on a $w:=v+e^n_{\sigma^{-1}(n)}\in [\![0,k-1]\!]^n_{\leq}$. Montrons qu'on a d'autre part 
$c\circ\sigma\in\mathfrak S_{|w^{-1}(0)|,\ldots,|w^{-1}(k-1)|}$, ceci dans chacun des cas $\sigma(n)\neq n$
et ($\sigma(n)=n$ et $v(n)<k-1$). 

{\it Premier cas : $\sigma(n)\neq n$.} Soit $\alpha$ l'indice tel que $\sigma^{-1}(n)\in I_\alpha$. On a nécessairement $\alpha<k$ et $n_\alpha>0$ ; de plus
$\sigma^{-1}(n)=\mathrm{max}(I_\alpha)$.
Rappelons que $(|v^{-1}(0)|,\ldots,|v^{-1}(k-1)|)=(n_1,\ldots,n_k)$, alors $(|w^{-1}(0)|,\ldots,|w^{-1}(k-1)|)
=(n_1,\ldots,n_\alpha-1,n_{\alpha+1}+1,\ldots,n_k)$, et la partition de $[\![1,n]\!]$ associée à $w$ est 
$(J_1,\ldots,J_k)$ avec $J_\beta=I_\beta$ pour $\beta\in [\![1,k]\!]-\{\alpha,\alpha+1\}$, 
$J_\alpha=I_\alpha-\{\mathrm{max}(I_\alpha)\}$, $J_{\alpha+1}=\{\mathrm{max}(I_\alpha)\}\cup I_{\alpha+1}$. Montrons que 
$c\circ\sigma\in\mathfrak S_{n_1,\ldots,n_\alpha-1,n_{\alpha+1}+1,\ldots,n_k}$.  
Si $\beta\in[\![1,k]\!]-\{\alpha\}$, alors $\sigma(I_\beta)\notni n$. La restriction de $c$ à 
$\sigma(I_\beta)$ est donc $x\mapsto x+1$, qui est strictement croissante. Donc la restriction de $c\circ\sigma$ à $I_\beta$ est 
strictement croissante. En particulier, si $\beta\in[\![1,k]\!]-\{\alpha,\alpha+1\}$, la restriction de $c\circ\sigma$ à $J_\beta=I_\beta$ 
est strictement croissante. 

Rappelons que la restriction de $c\circ\sigma$ à $I_{\alpha+1}$ est strictement croissante. On a $J_\alpha=
\{\mathrm{max}(I_\alpha)\}\cup I_{\alpha+1}$ avec $\mathrm{max}(I_\alpha)\leq I_{\alpha+1}$ et 
$c\circ\sigma(\mathrm{max}(I_\alpha))=0$, donc la restriction de $c\circ\sigma$ à $J_{\alpha+1}$ est strictement croissante. 

Comme $J_\alpha=I_\alpha-\{\mathrm{max}(I_\alpha)\}$ et que $\mathrm{max}(I_\alpha)=\sigma^{-1}(n)$, on a 
$\sigma(J_\alpha)\notni n$. La restriction de $c$ à $\sigma(J_\beta)$ est donc $x\mapsto x+1$, qui est strictement croissante. Donc la 
restriction de $c\circ\sigma$ à $J_\alpha$ est strictement croissante. 

Donc on a $c\circ\sigma\in\mathfrak S_{|w^{-1}(0)|,\ldots,|w^{-1}(k-1)|}$. 

{\it Deuxième cas : $\sigma(n)=n$ et $v(n)<k-1$.} Soit $p:=1+\mathrm{max}\{i\in[\![0,k-1]\!]|
|v^{-1}(i)|\neq 0\}$; alors $p<k$. Alors $(|v^{-1}(0)|,\ldots,|v^{-1}(k-1)|)=(n_1,\ldots,n_p,0,\ldots,0)$. 
Comme $\sigma(n)=n$, on a $\sigma^{-1}(n)=n=n_1+\cdots+n_p$. Donc 
$(|w^{-1}(0)|,\ldots,|w^{-1}(k-1)|)=(n_1,\ldots,n_p-1,1,0,\ldots,0)$, la partition 
correspondant à $w$ étant donnée par $(J_1,\ldots,J_k)$ avec $J_\alpha=I_\alpha$ pour $\alpha\neq p,p+1$, 
$J_p:=I_p-\{n\}$, et $J_{p+1}=\{n\}$. Si $\alpha\neq p,p+1$, $\sigma(I_\alpha)\notni n$, donc la restriction de 
$c$ à $\sigma(I_\alpha)$ est $x\mapsto x+1$ qui est strictement croissante, donc la restriction de $c\circ\sigma$
à $J_\alpha=I_\alpha$ est strictement croissante. On a $J_p=I_p-\{n\}$ et $\sigma(n)=n$, donc $\sigma(J_p)\notni n$, 
donc la restriction de $c$ à $\sigma(J_p)$ est $x\mapsto x+1$ qui est strictement croissante, donc la restriction de $c\circ\sigma$
à $J_p$ est strictement croissante. Enfin, $J_{p+1}$ est un singleton, donc la restriction de $c\circ\sigma$ à cet ensemble est 
strictement croissante.  

Donc $c\circ\sigma\in \mathfrak S_{n_1,\ldots,n_p-1,1,0,\ldots,0}=\mathfrak S_{|w^{-1}(0)|,\ldots,|w^{-1}(k-1)|}$. 

(c) Démonstration semblable à celle de (b). 
\end{proof}

\begin{lem}\label{lem:bij:BIS}
$\underline{\mathrm{bij}}$ induit une bijection $\mathrm{bij}$ entre $\mathrm{Ens}_{n-1}^k\times[\![0,n]\!]$
et $\{x\in \mathrm{Ens}_{n}^k\times[\![0,n]\!]|\mathrm{invol}(x)\notin\mathrm{Ens}_{n}^k\times[\![0,n]\!]\}$. On a 
alors le diagramme commutatif 
\begin{equation}\label{diag:c:0412}
\xymatrix{\mathrm{Ens}_{n-1}^k \times[\![0,n]\!]\ar@{^{(}->}[d]\ar^{
\!\!\!\!\!\!\!\!\!\!\!\!\!\!\!\!\!\!\!\!\!\!\!\!\!\!\!\!\!\!\!\!\!\!\!\!\!\!\!\!\!\!\!\!\!\!\!\!\!\!\!\!
\mathrm{bij}}_{
\!\!\!\!\!\!\!\!\!\!\!\!\!\!\!\!\!\!\!\!\!\!\!\!\!\!\!\!\!\!\!\!\!\!\!\!\!\!\!\!\!\!\!\!\!\!\!\!\!\!\!\!
\sim}[rr]&&\{x\in \mathrm{Ens}_n^k\times[\![0,n]\!]|
\mathrm{invol}(x)\notin\mathrm{Ens}_n^k\times[\![0,n]\!]\}
\}\ar@{^{(}->}[d]\\ 
\mathbb Z^{n-1}\times\mathfrak S_{n-1} \times[\![0,n]\!]
\ar_{\underline{\mathrm{bij}}}[rr]&&
\mathbb Z^n\times\mathfrak S_n \times[\![0,n]\!]}
\end{equation}
\end{lem}

\begin{proof}
D'après le lemme 1.29, il suffit de montrer que 
$\underline{\mathrm{bij}}$ induit une bijection entre $\mathrm{Ens}_{n-1}^k\times[\![0,n]\!]$
et 
$$
E:=E_n\sqcup E_0\sqcup E_{[\![1,n-1]\!]}
$$
où 
$$
E_n:=\{(v,\sigma,n)|(v,\sigma)\in \mathrm{Ens}_{n}^k\text{ et }\sigma(n)=n\text{ et }v(n)=k-1\},
$$
$$
E_0:=\{(v,\sigma,0)|(v,\sigma)\in \mathrm{Ens}_{n}^k\text{ et }\sigma(1)=1\text{ et }v(1)=0\},
$$
\begin{align*}
E_{[\![1,n-1]\!]}:= & \{(v,\sigma,i)\in \mathrm{Ens}_{n}^k\times[\![1,n-1]\!]|\sigma^{-1}(i)\notin
\{|v^{-1}(0)|,\ldots,|v^{-1}(0)|+\cdots+|v^{-1}(k-1)|\}
\\ & \text{ et }\sigma^{-1}(i+1)=\sigma^{-1}(i)+1\}. 
\end{align*}

Pour cela, on montre séparément que $\underline{\mathrm{bij}}$ induit une bijection entre 
(a) $\mathrm{Ens}_{n-1}^k\times\{n\}$ et $E_n$, (b) $\mathrm{Ens}_{n-1}^k\times\{0\}$ et $E_0$, 
et (c) $\mathrm{Ens}_{n-1}^k\times[\![1,n-1]\!]$ et $E_{[\![1,n-1]\!]}$. 

(a) Montrons que $\underline{\mathrm{bij}}$ envoie $\mathrm{Ens}_{n-1}^k\times\{n\}$ dans $E_n$. Soit $(w,\tau)\in\mathrm{Ens}_{n-1}^k$. 
On a $w\in[\![0,k-1]\!]_{\leq}^{n-1}$, ce qui implique $w^{(n)}=(w,k-1) \in [\![0,k-1]\!]_{\leq}^n$. Posons $v:=w^{(n)}$, alors 
pour tout $i \in [\![0,k-1]\!]$, on a $v^{-1}(i)=w^{-1}(i)$ si $i \neq k-1$ et $v^{-1}(k-1)=w^{-1}(k-1) \sqcup \{n\}$. 
Si $i \neq k-1$, la restriction de $\tau^{(n)}$ à $v^{-1}(i)$ coïncide avec la restriction de $\tau$ à $w^{-1}(i)$, qui est croissante
car $\tau\in\mathfrak S_{|w^{-1}(0)|,\ldots,|w^{-1}(k-1)|}$. La restriction de $\tau^{(n)}$ à $v^{-1}(k-1)$ est l'union disjointe de 
la restriction de $\tau$ à $w^{-1}(k-1)$, qui est croissante et à valeurs dans $[\![1,n-1]\!]$ et de l'application $n\mapsto n$. Comme 
on a $w^{-1}(k-1)<n$, cette union disjointe est croissante. Donc $\tau^{(n)}\in \mathfrak S_{|v^{-1}(0)|,\ldots,|v^{-1}(k-1)|}$, ce
qui implique $(w^{(n)},\tau^{(n)}) \in \mathrm{Ens}_n^k$. De plus, on a $w^{(n)}(n)=k-1$ et $\tau^{(n)}(n)=n$, donc 
$\underline{\mathrm{bij}}(w,\tau,n)=(w^{(n)},\tau^{(n)},n) \in E_n$. 

Donc $\underline{\mathrm{bij}}(\mathrm{Ens}_{n-1}^k\times\{n\})\subset E_n$, notons 
$$
\mathrm{bij}_n : \mathrm{Ens}_{n-1}^k\times\{n\}\to E_n
$$
l'application induite. Soit $\underline{\mathrm{invbij}}_n : E_n\to \mathbb Z^{n-1}\times\mathfrak S_{n-1}\times \{n\}$
l'application donnée par $(v,\sigma,n)\mapsto(v_{|[\![1,n-1]\!]},\sigma_{|[\![1,n-1]\!]},n)$ ; comme  $\sigma \in \mathfrak S_n$ et 
$\sigma(n)=n$, on a bien $\sigma_{|[\![1,n-1]\!]}\in \mathfrak S_{n-1}$. 

Montrons que $\underline{\mathrm{invbij}}_n$ envoie $E_n$ dans $\mathrm{Ens}_{n-1}^k\times\{n\}$. 
Soit $(v,\sigma,n)\in E_n$, et posons $w:=v_{|[\![1,n-1]\!]}$, $\tau:=\sigma_{|[\![1,n-1]\!]}$. 
Alors $v\in [\![0,k-1]\!]^n_{\leq}$ ce qui implique $w\in [\![0,k-1]\!]^n_{\leq}$. De plus, 
pour chaque $i\in [0,k-1]$, $w^{-1}(i)$ est contenu dans $w^{-1}(i)$ (on a même égalité si $i\neq k-1$).
La restriction de $\tau$ à $w^{-1}(i)$ coïncide avec la restriction de $\sigma$ au même ensemble, 
qui est croissante par la croissance de $\sigma$ en restriction à $v^{-1}(i)$, qui contient $w^{-1}(i)$. 
Donc $(w,\tau)\in \mathrm{Ens}_{n-1}^k$. 

Donc $\underline{\mathrm{invbij}}_n$ envoie $E_n$ dans $\mathrm{Ens}_{n-1}^k\times\{n\}$. Notons 
$$
\mathrm{invbij}_n : E_n\to\mathrm{Ens}_{n-1}^k\times\{n\}
$$
l'application induite. On vérifie que les compositions $\mathrm{invbij}_n\circ \mathrm{bij}_n$ et $\mathrm{bij}_n\circ \mathrm{invbij}_n$
sont l'identité. On en déduit que $\mathrm{bij}_n$ est une bijection. 

(b) Démonstration semblable à celle de (a).  

(c) Montrons que $\underline{\mathrm{bij}}$ envoie $\mathrm{Ens}_{n-1}^k \times[\![1,n-1]\!]$ dans $E_{[\![1,n-1]\!]}$. 
Soit $(w,\tau,i) \in \mathrm{Ens}_{n-1}^k \times [\![1,n-1]\!]$ et posons $v:=w \circ p_{i}$, $\sigma:=\tau^{(i)}$. 
Comme $p_{i} : [\![1,n]\!]\to[\![1,n-1]\!]$ est croissante et $w \in [\![0,k-1]\!]^{n-1}_{\leq}$, on a 
$v \in [\![0,k-1]\!]^n_{\leq}$. On a alors pour $j \in[\![0,k-1]\!]$, $v^{-1}(j)=w^{-1}(j)$ si $j<w(i)$, l'égalité
$v^{-1}(j)=w^{-1}(j)+1$ si $j>w(i)$, et $v^{-1}(w(i))=w^{-1}(w(i)) \cup  (w^{-1}(w(i))+1)$. Si $j<w(i)$, la restriction de 
$\sigma$ à $v^{-1}(j)=w^{-1}(j)$ est égale à la composition $\mathrm{st}_{\tau(i)} \circ \tau_{|w^{-1}(j)}$, qui est croissante 
par croissance de chacun des termes. Si $j>w(i)$, 
la restriction de $\sigma$ à $v^{-1}(j)=w^{-1}(j)+1$ est égale à la composition  $\mathrm{st}_{\tau(i)-1} \circ \tau_{|w^{-1}(j)} \circ 
(x\mapsto x-1)$, qui est croissante par croissance de chacun des termes. Notons $\alpha,\beta$ les éléments minimaux et maximaux de 
$w^{-1}(w(i))$ ; comme cet ensemble est un intervalle, on a $w^{-1}(w(i))=[\![\alpha,\beta]\!] \ni i$. Alors 
$v^{-1}(w(i))=[\![\alpha,\beta+1]\!]$. La restriction de $\sigma$ à $v^{-1}(w(i))=[\![\alpha,\beta+1]\!]$ est donnée par $x\mapsto \tau(x)$ 
si $x\in[\![\alpha,i]\!]$ et $x\mapsto\tau(x-1)+1$ si $x\in[\![i+1,\beta+1]\!]$ ; les restrictions de $\tau$ à 
$[\![\alpha,i]\!]$ et $[\![i,\beta]\!]$ sont croissantes, ce qui implique que les restrictions de $\sigma$ à 
$[\![\alpha,i]\!]$ et $[\![i+1,\beta+1]\!]$ le sont aussi.  On a $\sigma(i+1)=\tau(i)+1=\sigma(i)+1$ 
ce qui, combiné à la croissance de $\sigma$ sur $[\![\alpha,i]\!]$ et $[\![i+1,\beta+1]\!]$  
implique la croissance de $\sigma$ sur $[\![\alpha,\beta+1]\!]=v^{-1}(w(i))$. 
Donc $\sigma \in\mathfrak S_{|v^{-1}(0)|,\ldots,|v^{-1}(k-1)|}$ donc $(v,\sigma)=(w\circ p_i,\tau^{(i)}) \in \mathrm{Ens}_n^k$. 

On sait que $\tau^{(i)}(i)=\tau(i)$ et $\tau^{(i)}(i+1)=\tau(i)+1$. Alors $(\tau^{(i)})^{-1}(\tau(i))=i$. L'ensemble 
$\{|v^{-1}(0)|,\ldots,|v^{-1}(0)|+\cdots+|v^{-1}(k-1)|\}$ est celui des maxima des intervalles de la partition 
$[\![1,n]\!]=v^{-1}(0)\sqcup \ldots\sqcup v^{-1}(k-1)$. Celui de ces intervalles auquel appartient $i$ est 
$v^{-1}(v(i))=[\![\alpha,\beta+1]\!]$, on a donc pour $j\neq v(i)$, $i\neq\mathrm{max}(v(j))$. Comme $i\leq \beta$, on a 
$i\neq \beta+1=\mathrm{max}(v^{-1}(v(i)))$. On a donc pour tout $j\in[\![0,k-1]\!]$, $i\neq\mathrm{max}(v^{-1}(j))$, 
donc $(\tau^{(i)})^{-1}(\tau(i))=i\notin\{|v^{-1}(0)|,\ldots,|v^{-1}(0)|+\cdots+|v^{-1}(k-1)|\}$. De plus, on a 
$(\tau^{(i)})^{-1}(\tau(i)+1)=i+1=(\tau^{(i)})^{-1}(\tau(i))+1$. On en déduit $(w\circ p_i,\tau^{(i)},\tau(i))\in E_{[\![1,n-1]\!]}$. 

Donc $\underline{\mathrm{bij}}(\mathrm{Ens}_{n-1}^k \times[\![1,n-1]\!]) \subset E_{[\![1,n-1]\!]}$. Notons 
$$
\mathrm{bij}_{[\![1,n-1]\!]} : \mathrm{Ens}_{n-1}^k \times[\![1,n-1]\!]\to E_{[\![1,n-1]\!]}
$$
l'application induite.

Pour $(v,\sigma,j) \in E_{[\![1,n-1]\!]}$ posons 
\begin{equation}\label{***:1711:BIS}
    \underline{\mathrm{invbij}}_{[\![1,n-1]\!]}(v,\sigma,j):=(v \circ \mathrm{st}_{\sigma^{-1}(j)},
    p_j \circ \sigma \circ \mathrm{st}_{\sigma^{-1}(j)},\sigma^{-1}(j))=(w,\tau,i). 
\end{equation}
Alors $\mathrm{st}_{\sigma^{-1}(j)}$ est une application $[\![1,n-1]\!]\to[\![1,n]\!]$, donc $w \in \mathbb Z^{n-1}$. 
On a aussi $i \in [\![1,n]\!]$. Enfin $\tau=p_j \circ \sigma \circ \mathrm{st}_{\sigma^{-1}(j)}$ est une application 
de $[\![1,n-1]\!]$ dans lui-même. Montrons son injectivité. Supposons $x \neq x' \in [\![1,n-1]\!]$ et $\tau(x)=\tau'(x)$. 
Par injectivité de $\sigma$ et $\mathrm{st}_{\sigma^{-1}(j)}$ on a 
$\sigma \circ \mathrm{st}_{\sigma^{-1}(j)}(x) \neq \sigma \circ \mathrm{st}_{\sigma^{-1}(j)}(x')$ donc on a 
(quitte  à échanger $x$ et $x'$) $\sigma \circ \mathrm{st}_{\sigma^{-1}(j)}(x)=j$ 
et $\sigma \circ \mathrm{st}_{\sigma^{-1}(j)}(x')=j+1$. Donc $\mathrm{st}_{\sigma^{-1}(j)}(x)=\sigma^{-1}(j)$ et 
$\mathrm{st}_{\sigma^{-1}(j)}(x)=\sigma^{-1}(j)+1$. Or $\sigma^{-1}(j)+1$ n'est pas dans l'image de $\mathrm{st}_{\sigma^{-1}(j)}$, 
contradiction. On a donc l'injectivité de $\tau$, donc $\tau\in\mathfrak S_{n-1}$. Ceci montre que \eqref{***:1711:BIS} définit une application 
$$
\underline{\mathrm{invbij}}_{[\![1,n-1]\!]} :  E_{[\![1,n-1]\!]}\to\mathbb Z^n \times\mathfrak S_n \times [\![1,n-1]\!]. 
$$
Montrons que $\underline{\mathrm{invbij}}_{[\![1,n-1]\!]}$ envoie $E_{[\![1,n-1]\!]}$ dans 
$\mathrm{Ens}_n^k \times[\![1,n-1]\!]$. 
Dans \eqref{***:1711:BIS}, on a $w=v \circ \mathrm{st}_{\sigma^{-1}(j)}$. On a $v \in [\![0,k-1]\!]^{n-1}_{\leq}$ 
et $\mathrm{st}_{\sigma^{-1}(j)}$ est une application croissante $[\![1,n-1]\!]\to[\![1,n]\!]$, ce qui implique 
$w \in [\![1,k-1]\!]^n_{\leq}$. 

Montrons que $\tau\in\mathfrak S_{|w^{-1}(0)|,\ldots,|w^{-1}(k-1)|}$. Ceci revient à montrer que la restriction de $\tau$ à 
$w^{-1}(\ell)$ est croissante pour tout $\ell\in[\![0,k-1]\!]$. Rappelons que $v^{-1}(v(\sigma^{-1}(i)))$ est 
un intervalle de $[\![1,n]\!]$ dont $\sigma^{-1}(i)$ n'est pas le plus grand élément. Notons $\alpha,\beta$ les minimum et maximum 
de cet intervalle, alors $v^{-1}(v(\sigma^{-1}(i))) =[\![\alpha,\beta]\!]$ avec $\alpha \leq \sigma^{-1}(i)<\beta$. 

Si $\ell<v(\sigma^{-1}(i))$, alors $w^{-1}(\ell)=v^{-1}(\ell)$, et la restriction de $\tau$ à $w^{-1}(\ell)$ coïncide avec la restriction 
de $p_{\sigma^{-1}(j)}\circ\sigma$ au même intervalle. Comme $\sigma_{|v^{-1}(\ell)}$ et $p_{\sigma^{-1}(j)}$ sont croissantes, 
$\tau_{|w^{-1}(\ell)}$ est donc croissante. Si $\ell>v(\sigma^{-1}(i))$, alors $w^{-1}(\ell)=v^{-1}(\ell)-1$, et la restriction de 
$\tau$ à $w^{-1}(\ell)$ coïncide avec la restriction
de $p_{\sigma^{-1}(j)} \circ \sigma \circ (x\mapsto x+1)$ au même intervalle. Comme 
$(\sigma \circ (x\mapsto x+1))_{|v^{-1}(\ell)-1}$ et $p_{\sigma^{-1}(j)}$ sont croissantes, $\tau_{|w^{-1}(\ell)}$ est donc croissante. On a 
$w^{-1}(v(\sigma^{-1}(i)))=[\![\alpha,\beta-1]\!]$. La restriction de $\tau$ à $[\![\alpha,\sigma^{-1}(i)]\!]$ coïncide avec celle de 
$\sigma$ à cet intervalle qui est croissante, cet intervalle étant contenu dans $v^{-1}(v(\sigma^{-1}(i)))$, donc $\tau_{|[\![\alpha,\sigma^{-1}(i)]\!]}$ est croissante. La restriction de $\tau$ à 
$[\![\sigma^{-1}(i),\beta-1]\!]$ coïncide avec celle de $\sigma \circ (x\mapsto x+1)$ à cet intervalle qui est croissante
par croissance de $\sigma$ sur $v^{-1}(v(\sigma^{-1}(i)))$, donc 
$\tau_{|[\![\sigma^{-1}(i),\beta-1]\!]}$ est croissante. Il suit que la restriction de $\tau$ à 
$[\![\alpha,\beta-1]\!]=w^{-1}(v(\sigma^{-1}(i)))$ est croissante. Tout ceci implique 
$\tau \in \mathfrak S_{|w^{-1}(0)|,\ldots,|w^{-1}(k-1)|}$. 

Comme $\tau \in \mathfrak S_{|w^{-1}(0)|,\ldots,|w^{-1}(k-1)|}$ et $w \in [\![1,k-1]\!]^n_{\leq}$, on a
$(w,\tau) \in \mathrm{Ens}_{n-1}^k$. 

Donc 
$\underline{\mathrm{invbij}}_{[\![1,n-1]\!]}(E_{[\![1,n-1]\!]}) \subset \mathrm{Ens}_n^k \times[\![1,n-1]\!]$.
Notons 
$$
\mathrm{invbij}_{[\![1,n-1]\!]} : E_{[\![1,n-1]\!]}\to \mathrm{Ens}_n^k \times[\![1,n-1]\!]
$$
l'application induite. On vérifie que les compositions $\mathrm{invbij}_n\circ \mathrm{bij}_n$ et $\mathrm{bij}_n\circ \mathrm{invbij}_n$
sont l'identité. On en déduit que $\mathrm{bij}_{[\![1,n-1]\!]}$ est une bijection. 
\end{proof}

\subsubsection{Diagramme commutatif impliquant $\mathrm{Aff}(\Delta^{n-1},\Delta^n)\times\{\pm1\}$}

\begin{defn}\label{def:A:2210:BIS}
Pour $n,m\geq 0$, on note $\mathrm{Aff}(\Delta^n,\Delta^m)$ l'ensemble des applications $\phi : \Delta^n\to\Delta^m$ telles 
qu'il existe une application affine $\underline \phi : \mathbb R^n\to\mathbb R^m$ telle que 
$\underline \phi\circ\mathrm{can}_n=\mathrm{can}_m\circ \phi$. 
(avec $\mathrm{can}_n$, $\mathrm{can}_m$ comme en déf. \ref{def:2:1:2311})\end{defn}

\begin{lem}\label{lem:A:2210:BIS}
Soit $n,m\geq 0$. 

(a) Si $\phi\in\mathrm{Aff}(\Delta^n,\Delta^m)$, une application affine $\underline \phi:\mathbb R^n\to\mathbb R^m$ comme en 
déf. \ref{def:A:2210:BIS} est unique. 

(b) L'application $\mathrm{Aff}(\Delta^n,\Delta^m)\to\mathbf{Aff}(\mathbb R^n,\mathbb R^m)$, $\phi\mapsto\underline \phi$ est injective. 
\end{lem}

\begin{proof}
(a) est une conséquence de ce que $\Delta^n$ contient une base affine de $\mathbb R^n$, à savoir $(E_0^n,\ldots,E_n^n)$. 
(b) provient de ce qu'une application affine est uniquement déterminée par l'image d'une base affine. 
\end{proof}

On déduit du lemme \ref{lem:A:2210:BIS} une famille d'inclusions 
\begin{equation}\label{incl:Delta:R:0412}
\mathrm{Aff}(\Delta^n,\Delta^m) \subset \mathbf{Aff}(\mathbb R^n,\mathbb R^m)
\end{equation}
pour $n,m\geq 0$, compatible avec les compositions d'applications. 

\begin{lemdef}
Dans la situation de la déf. \ref{def:1709:2311:BIS}, si $P_0,\ldots,P_n\in\Delta^m$, l'application 
$[\underline{P_0,\ldots,P_n}]$ envoie $\Delta^n$ dans $\Delta^m$, donc appartient à l'image de 
l'injection $\mathrm{Aff}(\Delta^n,\Delta^m)\hookrightarrow\mathbf{Aff}(\mathbb R^n,\mathbb R^m)$. On note 
$[P_0,\ldots,P_n]$ l'élément de $\mathrm{Aff}(\Delta^n,\Delta^m)$ correspondant. 
\end{lemdef}

\begin{proof}
Suit de la convexité de $\Delta^n$ et $\Delta^m$, et de ce que $(E_0^n,\ldots,E_n^n)$ forme une base du convexe 
$\Delta^n$. 
\end{proof}

\begin{lemdef}\label{lem:def:0412}
Si $(v,\sigma)\in\mathrm{Ens}_n^k$, alors $\underline c_k(v,\sigma)$ appartient à l'image de l'injection 
$\mathrm{Aff}(\Delta^n,\Delta^n)\hookrightarrow\mathbf{Aff}(\mathbb R^n,\mathbb R^n)$. On note 
$c_k(v,\sigma)\in\mathrm{Aff}(\Delta^n,\Delta^n)$ la préimage de $\underline c_k(v,\sigma)$. 
\end{lemdef}

\begin{proof} L'application associée à $(v,\sigma)\in\mathrm{Ens}_n^k$ est donnée par 
$$
(x_1,\ldots,x_n)\mapsto ((v_1+x_{\sigma(1)})/k,\ldots,(v_n+x_{\sigma(n)})/k)=:(y_1,\ldots,y_n) ; 
$$
c'est une endo-application continue de $\mathbb R^n$. 
Supposons $(x_1,\ldots,x_n)\in\Delta^n$. Comme $v_1\geq0$ et $x_{\sigma(1)}\geq0$, on a $y_1\geq0$ ; de même, $v_n\leq k-1$ et 
$x_{\sigma(n)}\leq1$ implique $y_n\leq 1$. Enfin pour $i\in[\![1,n-1]\!]$, la déf. \ref{def:ens:n:k:1210:BIS} implique que 
$v_i=v_{i+1}$ ou $v_i<v_{i+1}$. Dans le premier cas ($v_i=v_{i+1}$), \eqref{TOTO:2311:BIS} 
implique $\sigma(i)<\sigma(i+1)$, ce qui implique par croissance de $i\mapsto x_i$ l'inégalité $x_{\sigma(i)}\leq x_{\sigma(i+1)}$, ce qui 
combiné avec $v_i=v_{i+1}$ implique l'inégalité dans $y_i=(v_i+x_{\sigma(i)})/k\leq (v_{i+1}+x_{\sigma(i+1)})/k=y_{i+1}$, où les égalités
extrêmes proviennent des définitions. Dans le second cas ($v_i<v_{i+1}$), on a  $y_i=(v_i+x_{\sigma(i)})/k\leq(v_i+1)/k\leq v_{i+1}/k 
\leq (v_{i+1}+x_{\sigma(i+1)})/k=y_{i+1}$, où la première et dernière égalité proviennent des définitions, où la première et dernière
inégalité proviennent respectivement de $x_{\sigma(i)}\geq 0$ et $x_{\sigma(i+1)}\leq 1$, et où l'inégalité centrale provient de 
$v_i<v_{i+1}$ et du caractère entier des composantes de $v$. On a donc dans tous les cas $y_i\leq y_{i+1}$, ce qui achève de montrer que 
$(y_1,\ldots,y_n)\in\Delta^n$. Ceci montre l'énoncé. 
\end{proof}

On a $c_k(v,\sigma)=[(1/k)(v+\sigma^*E^n_0),\ldots,(1/k)(v+\sigma^*E^n_n)]$ pour $(v,\sigma)\in\mathrm{Ens}_n^k$. 

\begin{defn}
On note pour $i\in[\![0,n]\!]$, 
$$
\partial_i^n:=[E_0^n,\ldots,E_{n-i-1}^n,E_{n-i+1}^n,\ldots,E_n^n]\in\mathrm{Aff}(\Delta^{n-1},\Delta^n). 
$$
\end{defn}

\begin{lem}\label{lem:lien:2210}
$\underline\partial_i^n$ est l'image de $\partial_i^n$  par l'application 
$\mathrm{Aff}(\Delta^{n-1},\Delta^n)\to\mathbf{Aff}(\mathbb R^{n-1},\mathbb R^n)$. 
\end{lem}

\begin{proof}
Immédiat. 
\end{proof}

\begin{defn}\label{titi:0512}
On note 
$$
f : \mathrm{Ens}_n^k \times[\![0,n]\!]\to \mathrm{Aff}(\Delta^{n-1},\Delta^n),\quad 
\tilde f : \mathrm{Ens}_{n-1}^k \times[\![0,n]\!]\to \mathrm{Aff}(\Delta^{n-1},\Delta^n)
$$
les applications données par 
$$
f(v,\sigma,i):=c_k(v,\sigma)\circ\partial_i^n,\quad 
\tilde f(w,\tau,i):=\partial_i^n\circ c_k(w,\tau).  
$$
\end{defn}

\begin{lem}\label{lem:2:33:0512}
Les diagrammes suivants commutent
$$
\xymatrix
{\mathrm{Ens}_n^k \times[\![0,n]\!]\ar^f[r]\ar@{^{(}->}[d]&\mathrm{Aff}(\Delta^{n-1},\Delta^n)\ar@{^{(}->}[d]\\
\mathbb Z^n\times\mathfrak S_n\times\times[\![0,n]\!]\ar_{\underline f}[r]&
\mathbf{Aff}(\mathbb R^{n-1},\mathbb R^n)}
\quad
\xymatrix
{\mathrm{Ens}_{n-1}^k \times[\![0,n]\!]\ar^{\tilde f}[r]\ar@{^{(}->}[d]&\mathrm{Aff}(\Delta^{n-1},\Delta^n)\ar@{^{(}->}[d]\\
\mathbb Z^{n-1}\times\mathfrak S_{n-1}\times\times[\![0,n]\!]\ar_{\underline{\tilde f}}[r]&
\mathbf{Aff}(\mathbb R^{n-1},\mathbb R^n)}
$$
\end{lem}

\begin{proof}
Cela provient de la compatibilité des applications \eqref{incl:Delta:R:0412} avec la composition, du lemme \ref{lem:lien:2210},
et ce que pour $(v,\sigma)\in\mathrm{Ens}_n^k$ (resp. $(w,\tau)\in\mathrm{Ens}_{n-1}^k$), l'image de 
$c_k(v,\sigma)$ (resp. $c_k(w,\tau)$) sous $\mathrm{Aff}(\Delta^n,\Delta^n)\to\mathbf{Aff}(\mathbb R^n,\mathbb R^n)$
(resp. $\mathrm{Aff}(\Delta^{n-1},\Delta^{n-1})\to\mathbf{Aff}(\mathbb R^{n-1},\mathbb R^{n-1})$) est 
$\underline c_k(v,\sigma)$ (resp. $\underline c_k(w,\tau)$) (cf. lemme \ref{lem:def:0412}). 
\end{proof}

\begin{lem}\label{compats:f:underlinef:0412}
Le diagramme suivant commute 
\begin{equation}\label{diag:d:0412}
\xymatrix{
\mathrm{Ens}_{n-1}^k \times[\![0,n]\!]\ar^{\mathrm{bij}}_{\sim}[rr]\ar_{(\tilde f,\widetilde{\mathrm{sgn}})}[rd]&&
\{x\in \mathrm{Ens}_n^k\times[\![0,n]\!]|
\mathrm{invol}(x)\notin\mathrm{Ens}_n^k\times[\![0,n]\!]\}
\} \ar^{(f,\mathrm{sgn})}[ld]
\\ 
&\mathrm{Aff}(\Delta^{n-1},\Delta^n)\times\{\pm1\}&
}
\end{equation}
\end{lem}

\begin{proof}
Ceci provient de la commutativité des diagrammes \eqref{diag:b:0412} et \eqref{diag:c:0412}  
ainsi que du lemme \ref{compats:f:underlinef:0412}.
\end{proof}

\subsection{Construction d'endomorphismes de groupes de chaînes}\label{sect:2:4:2512}

Les résultats de cette section seront utilisés en section \ref{1214:2111} afin de montrer le (b) du théorème \ref{thm:ppal}. 

\subsubsection{Morphismes dans une catégorie $\mathcal C$}

\begin{defn}
Soit $\mathcal C$ la petite catégorie dont l'ensemble d'objets est $\mathbb Z_{\geq 0}$, avec 
$\mathcal C(n,m):=\mathbb Z\mathbf{Top}(\Delta^n,\Delta^m)$, et dont la composition est donnée par la
linéarisation de la composition dans $\mathbf{Top}$. 
\end{defn}

Toute application affine étant continue, on a une famille de diagrammes 
$$
\mathrm{Aff}(\Delta^n,\Delta^m) \subset \mathcal{C}(n,m)
$$
pour $n,m\geq 0$, compatible avec les compositions d'applications. 

\begin{defn}\label{def:div:n:k}
On pose  
$$
\mathrm{div}_n^k:=\sum_{(v,\sigma)\in\mathrm{Ens}_n^k}\epsilon(\sigma)c_k(v,\sigma)\in
\mathbb Z\mathrm{Aff}(\Delta^n,\Delta^n)\subset \mathcal C(n,n). 
$$
\end{defn}

\begin{defn}\label{def:toto:0512}
On note 
$$
\partial_{n-1,n}:=\sum_{i=0}^n(-1)^i\partial_i^n\in\mathbb Z\mathrm{Aff}(\Delta^{n-1},\Delta^n)\subset \mathcal C(n-1,n). 
$$
\end{defn}

On sait que $\partial_{n,n+1}\circ\partial_{n-1,n}=0$. De plus, pour tout $p\geq 0$, le complexe 
$\cdots\to \mathcal C(k,p)\stackrel{-\circ\partial_{k-1,k}}{\to} \mathcal C(k-1,p)\to\cdots\to C(0,p)\to\mathbb Z\to 0$
est acyclique (l'homologie de $\Delta^p$ étant donnée par $\mathrm H_k(\Delta^p)=0$ si $k>0$ et $=\mathbb Z$ si $k=0$).

\subsubsection{Démonstration de $\mathrm{div}_n^k\circ \partial_{n-1,n}=\partial_{n-1,n}\circ\mathrm{div}_{n-1}^k$}\label{sect;demo}

Dans la section \ref{sect;demo}, on fixe $n,k\geq1$. 

On a  
\begin{equation}\label{eq:A:2210}
\mathrm{div}_n^k\circ\partial_{n-1,n}
=\sum_{x\in\mathrm{Ens}_n^k \times[\![0,n]\!]}\mathrm{sgn}(x)f(x). 
\end{equation}
en combinant les défs. \ref{def:2:9:0512}, \ref{titi:0512}, \ref{def:div:n:k}, \ref{def:toto:0512},  
et 
\begin{equation}\label{eq:B:2210}
\partial_{n-1,n}\circ\mathrm{div}_{n-1}^k
=\sum_{\tilde x\in\mathrm{Ens}_{n-1}^k \times[\![0,n]\!]}\widetilde{\mathrm{sgn}}(\tilde x)\tilde f(\tilde x). 
\end{equation}
en combinant les défs. \ref{def:2:17:0512} 2.17, \ref{titi:0512}, \ref{def:div:n:k}, \ref{def:toto:0512}
(égalités dans $\mathbb Z\mathrm{Aff}(\Delta^{n-1},\Delta^n)$). 



\begin{prop}\label{prop;2110}
On a pour tous $n,k\geq1$
$$
\mathrm{div}_n^k\circ \partial_{n-1,n}=\partial_{n-1,n}\circ\mathrm{div}_{n-1}^k 
$$
(égalité dans $\mathcal C(n-1,n)$). 
\end{prop}

\begin{proof} 
On a 
\begin{align*}
    &\sum_{\substack{x\in\mathrm{Ens}_n^k \times[\![0,n]\!]| \\ \mathrm{invol}(x)\in \mathrm{Ens}_n^k \times[\![0,n]\!]}}
    \mathrm{sgn}(x)f(x)
    =\sum_{\substack{x\in\mathrm{Ens}_n^k \times[\![0,n]\!]| \\ \mathrm{invol}(x)\in \mathrm{Ens}_n^k \times[\![0,n]\!]}}
    \mathrm{sgn}(\mathrm{invol}(x))f(\mathrm{invol}(x))
    \\& =-
    \sum_{\substack{x\in\mathrm{Ens}_n^k \times[\![0,n]\!]|\\ \mathrm{invol}(x)\in \mathrm{Ens}_n^k \times[\![0,n]\!]}}
    \mathrm{sgn}(x)f(x)
\end{align*}
où la première égalité suit de ce que $\mathrm{invol}$ est une involution de $\{x\in\mathrm{Ens}_n^k 
\times[\![0,n]\!]|\mathrm{invol}(x)\in \mathrm{Ens}_n^k \times[\![0,n]\!]\}$ (cf. lemme \ref{lem:inv:BIS}(a)) 
et la deuxième égalité suit du lemme \ref{lem:inv:BIS}(b) et de la première égalité du lemme \ref{lem:2:33:0512}. 
On en déduit 
\begin{equation}\label{22:44:2110}
\sum_{\substack{x\in\mathrm{Ens}_n^k \times[\![0,n]\!]|\\ \mathrm{invol}(x)\in \mathrm{Ens}_n^k \times[\![0,n]\!]}}\mathrm{sgn}(x)f(x)=0. 
\end{equation}
Alors  
\begin{align*}
& \mathrm{div}_n^k\circ\partial_{n-1,n}
=\sum_{x\in\mathrm{Ens}_n^k \times[\![0,n]\!]}\mathrm{sgn}(x)f(x)
\\ & 
=\sum_{\substack{x\in\mathrm{Ens}_n^k \times[\![0,n]\!]|\\ \mathrm{invol}(x)\in \mathrm{Ens}_n^k \times[\![0,n]\!]}}\mathrm{sgn}(x)f(x)
+\sum_{\substack{x\in\mathrm{Ens}_n^k \times[\![0,n]\!]|\\ \mathrm{invol}(x)\notin \mathrm{Ens}_n^k \times[\![0,n]\!]}}\mathrm{sgn}(x)f(x)
\\ & 
=\sum_{\substack{x\in\mathrm{Ens}_n^k \times[\![0,n]\!]|\\ \mathrm{invol}(x)\notin \mathrm{Ens}_n^k \times[\![0,n]\!]}}\mathrm{sgn}(x)f(x)
=\sum_{y\in\mathrm{Ens}_{n-1}^k \times[\![0,n]\!]}\mathrm{sgn}(\mathrm{bij}(y))f(\mathrm{bij}(y))
\\ & =\sum_{y\in\mathrm{Ens}_{n-1}^k \times[\![0,n]\!]}\widetilde{\mathrm{sgn}}(y)\tilde f(y)
=\partial_{n-1,n}\circ\mathrm{div}_{n-1}^k
\end{align*}
(égalité dans $\mathbb Z\mathrm{Aff}(\Delta^{n-1},\Delta^n)$) où la première égalité suit de \eqref{eq:A:2210}, la troisième 
égalité suit de \eqref{22:44:2110}, la quatrième égalité suit du lemme \ref{lem:bij:BIS}, la cinquième égalité suit du lemme 
\ref{lem:composition:BIS} et des deux égalités du lemme \ref{lem:2:33:0512}, la sixième égalité suit de \eqref{eq:B:2210}. 

On en déduit l'égalité annoncée, les applications 
$\mathrm{Aff}(\Delta^n,\Delta^m)\to\mathcal C(n,m)$ étant compatibles aux compositions. 
\end{proof}

\subsubsection{Relation dans $\mathcal C$ entre $\mathrm{div}_\bullet^k$ et $id_\bullet$}

\begin{lem}
Pour tout $k\geq 0$,  il existe une famille $(L^k_{n+1,n})_{n\geq 0}$ avec $L^k_{n+1,n}\in\mathcal C(n+1,n)$, telle que 
pour tout $n\geq 0$, on a
\begin{equation}\label{hmtp:id}
\mathrm{id}_n-\mathrm{div}_n^k=L^k_{n+1,n}\circ\partial_{n,n+1}+\partial_{n-1,n}\circ L^k_{n,n-1}. 
\end{equation}
\end{lem}

\begin{proof}
Montrons par récurrence sur $n\geq 0$ l'existence d'une famille $(L_{m+1,m}^k)_{m\leq n}$ telle que pour tout 
$m\leq n$, on a $\mathrm{id}_n-\mathrm{div}_n^k=L^k_{n+1,n}\circ\partial_{n,n+1}+\partial_{n-1,n}\circ L^k_{n,n-1}$ (énoncé $E(n)$).

Posons $L_{1,0}^k:=0$, alors on a $\mathrm{id}_0-\mathrm{div}_0^k=L^k_{1,0}\circ\partial_{0,1}$ d'où l'énoncé 
$E(0)$. 

Soit $n\geq 1$, et supposons $E(n-1)$ vérifié avec une famille $(L_{m+1,m}^k)_{m\leq n-1}$. Alors 
\begin{align*}
& (\mathrm{id}_n-\mathrm{div}_n^k-\partial_{n-1,n} \circ L_{n,n-1}^k) \circ \partial_{n-1,n}
=\partial_{n-1,n}-\mathrm{div}_n^k \circ \partial_{n-1,n}-\partial_{n-1,n} \circ L_{n,n-1}^k \circ \partial_{n-1,n}
\\ & 
=\partial_{n-1,n}-\partial_{n-1,n} \circ \mathrm{div}_{n-1}^k-\partial_{n-1,n} \circ L_{n,n-1}^k \circ \partial_{n-1,n} 
=\partial_{n-1,n} \circ (\mathrm{id}_{n-1}-\mathrm{div}_{n-1}^k-L_{n,n-1}^k \circ \partial_{n-1,n})
\\ & 
=\partial_{n-1,n} \circ (\mathrm{id}_{n-1}-\mathrm{div}_{n-1}^k-L_{n,n-1}^k \circ \partial_{n-1,n}-\partial_{n-2,n-1} \circ L_{n-1,n-2}^k) 
=\partial_{n-1,n} \circ0=0
\end{align*}
où la seconde égalité suit de la proposition \ref{prop;2110}, la quatrième égalité suit de $\partial_{n-1,n} \circ \partial_{n-2,n-1}=0$, 
la cinquième égalité suit de $E(n-1)$. 

Donc $\mathrm{id}_n-\mathrm{div}_n^k-\partial_{n-1,n} \circ L_{n,n-1}^k$ appartient au noyau de l'application $-\circ \partial_{n-1,n} : 
\mathcal C(n,n)\to \mathcal C(n-1,n)$, qui par acyclicité est égal à l'image de l'application 
$-\circ \partial_{n,n+1} : \mathcal C(n+1,n)\to \mathcal C(n,n)$. Il existe donc $L^k_{n+1,n}\in\mathcal C(n+1,n)$ tel que 
$\mathrm{id}_n-\mathrm{div}_n^k-\partial_{n-1,n} \circ L_{n,n-1}^k=L_{n+1,n}^k\circ\partial_{n,n+1}$, ce qui implique 
$E(n)$. 
\end{proof}

\subsubsection{Endomorphismes de groupes de chaînes singulières}

Soit $X$ un espace topologique. Pour $n\geq 0$, on note $C_n(X):=\mathbb Z\mathbf{Top}(\Delta^n,X)$. Pour $n,m\geq0$, on a une 
application $C_n(X)\times \mathcal C(m,n)\to C_m(X)$ induite par la composition $(c,x)\mapsto c\circ x$. 

\begin{defn}\label{def:x^*}
Pour $x\in \mathcal C(m,n)$, on note $x^* : C_n(X)\to C_m(X)$ l'application $c\mapsto c\circ x$. 
\end{defn}

Alors $(x\circ y)^*=y^*\circ x^*$. 

\begin{lem}\label{lem:images:source:2809}
Si $Y$ est un sous-ensemble de $X$, 
on a $x^*(C_n(Y))\subset C_m(Y)$. 
\end{lem}

\begin{proof}
Provient de ce que $x^*$ est une composition à la source. 
\end{proof}

De plus pour $n\geq 1$,  
$\partial_{n-1,n}^* : C_n(X)\to C_{n-1}(X)$ coïncide avec la différentielle singulière $\partial_n$. 

Alors \eqref{hmtp:id} implique 
\begin{equation}\label{132:2809}
\mathrm{id}_{C_n(X)}-(\mathrm{div}_n^k)^*=\partial_{n,n+1}^*\circ (L^k_{n+1,n})^*+(L^k_{n,n-1})^*\circ\partial_{n-1,n}^* 
\end{equation}
(égalité d'endomorphismes de $C_n(X)$) pour tout $n\geq0$. 

\subsection{Démonstration de (b) du théorème \ref{thm:ppal}}\label{1214:2111}

On se place dans le cadre de la section \ref{sect:res:ppal}: $X$ est un espace topologique et $a,b\in X$. 

\subsubsection{Composition de chemins}

\begin{defn}\label{def:appl:aff}
Si $s,t,u,v\in\mathbb R$ avec $s\neq t$, on note $a_{s,t}^{u,v}$ l'unique application affine de $\mathbb R$ dans lui-même telle que 
$s\mapsto u$ et $t\mapsto v$. 
\end{defn}

Soit $X$ un espace topologique, soit $a_0,\ldots,a_m\in X$ et $\tilde\gamma_i\in\mathrm{Chem}(a_i,a_{i+1})$ pour $i\in[\![0,m-1]\!]$. 

\begin{defn}
$\tilde\gamma_{m-1}*\ldots*\tilde\gamma_0\in\mathrm{Chem}(a_0,a_m)$ est le chemin tel que pour pour $i\in[\![0,m-1]\!]$, la restriction  
$(\tilde\gamma_{m-1}*\ldots*\tilde\gamma_0)_{|[i/m,(i+1)/m]}$ à $[i/m,(i+1)/m]$ coïncide avec $\tilde\gamma_i\circ a_{i/m,(i+1)/m}^{0,1}$
(conditions cohérentes car $\tilde\gamma_i(1)=\tilde\gamma_{i+1}(0)$ pour $i\in[\![0,m-2]\!]$).
\end{defn}

On a alors, en notant $[-]$ l'application canonique $\mathrm{Chem}(a,b)\to\pi_1(X;a,b)$ pour $a,b\in X$ quelconques, l'égalité 
\begin{equation}\label{eq:2106:2505}
[\tilde\gamma_{n-1}*\ldots*\tilde\gamma_0]=[\tilde\gamma_{n-1}]\cdots[\tilde\gamma_0]    
\end{equation}
(égalité dans $\pi_1(X;a_0,a_m)$, le produit dans le membre de droite étant celui dans le groupoïde $\pi_1(X)$). 

\subsubsection{Calcul de $(\mathrm{div}_n^k)^*((\tilde\gamma_k*\cdots*\tilde\gamma_1)^{(n)})$}\label{calcul:2809}

\begin{lem}\label{*:1210}
Il existe une unique application 
\begin{equation}\label{bij:1210}
\{(n_1,\ldots,n_k)|n_1\geq 0,\ldots,n_k\geq 0\text{ et }n_1+\cdots+n_k=n\}\to[\![0,k-1]\!]^n
\end{equation}
qui envoie $(n_1,\ldots,n_k)$ vers l'élément $v\in [\![0,k-1]\!]^n$ tel que pour $i=0,\ldots,k-1$ on a 
$v_{|n_1+\cdots+n_i+[\![1,n_{i+1}]\!]}=i$. Cette application est bijective, et la bijection réciproque envoie $v\in[\![0,k-1]\!]^n$ 
vers $(n_1,\ldots,n_k)$ donné par $n_i=|v^{-1}(i-1)|$ pour $i\in[\![1,k]\!]$. 
\end{lem}

\begin{proof}
Immédiate. 
\end{proof}

\begin{defn}
Pour $v\in [\![0,k-1]\!]^n$, on pose $\mathfrak S(v):=\{\sigma\in\mathfrak S_n|(v,\sigma)\in\mathrm{Ens}_n^k\}$ (cf. déf. 
\ref{def:ens:n:k:1210:BIS}).
\end{defn}

\begin{lem}\label{*a*1210}
Si $(n_1,\ldots,n_k)\in\{(n_1,\ldots,n_k)|n_1\geq 0,\ldots,n_k\geq 0$ et $n_1+\cdots+n_k=n\}$ et si $v$ est l'image de 
$(n_1,\ldots,n_k)$ par la bijection \eqref{bij:1210}, alors on a 
$\mathfrak S_{n_1,\ldots,n_k}=\mathfrak S(v)$. 
\end{lem}

\begin{proof}
L'application $v:[\![1,n]\!]\to[\![0,k-1]\!]$ étant constante sur chaque sous-ensemble $n_1+\cdots+n_i+[\![1,n_{i+1}]\!]$ 
et les valeurs prises sur des sous-ensembles consécutifs étant strictement croissantes, on a pour $i\in[\![1,n]\!]$ l'équivalence 
\begin{equation}\label{(equiv):1210}
(v_i<v_{i+1})\iff (i\in \{n_1,n_1+n_2,\ldots,n_1+\cdots+n_{k-1}\}).     
\end{equation} 
Soit alors $\sigma\in\mathfrak S_n$. On a $\sigma\in\mathfrak S(v)$ si et seulement si 
$$
\forall i\in[\![1,n]\!], \quad (\sigma(i)>\sigma(i+1))\implies
(v_i<v_{i+1}) ; 
$$
d'après \eqref{(equiv):1210} une condition équivalente est 
$$
\forall i\in[\![1,n]\!], \quad (\sigma(i)>\sigma(i+1))\implies(i\in \{n_1,n_1+n_2,\ldots,n_1+\cdots+n_{k-1}\}).  
$$
Donc $\sigma\in\mathfrak S(v)$ si et seulement si $\sigma$ est croissante sur les sous-ensembles 
$[\![1,n_1]\!]$, $n_1+[\![1,n_2]\!]$, etc., $n_1+\cdots+n_{k-1}+[\![1,n_k]\!]$, c'est à dire si et seulement si $\sigma\in\mathfrak S_{n_1,\ldots,n_k}$. 
\end{proof}

\begin{defn}
On pose $\widetilde{\mathrm{Ens}}_n^k:=\{((n_1,\ldots,n_k),\sigma)|n_1\geq 0,\ldots,n_k\geq 0$ et 
$n_1+\cdots+n_k=n$ et $\sigma\in\mathfrak S_{n_1,\ldots,n_k}\}$. 
\end{defn}

\begin{lem}\label{*b*1210}
Il existe une unique bijection 
\begin{equation}\label{nllebij:1210}
\mathrm{Ens}_n^k\to\widetilde{\mathrm{Ens}}_n^k    
\end{equation}
qui envoie $((n_1,\ldots,n_k),\sigma)$ vers 
$(v,\sigma)$, où $v$ est l'image de $(n_1,\ldots,n_k)$ par \eqref{bij:1210}. 
\end{lem}

\begin{proof}
Conséquence des lemmes \ref{*:1210} et \ref{*a*1210}.  
\end{proof}

Notons que pour $n_1+\cdots+n_k=n$ et $\sigma\in\mathfrak S_{n_1,\ldots,n_k}$, l'automorphisme de $[0,1]^n$ dans la catégorie $\mathbf{Top}$ 
donné par $\sigma^*$ induit un élément, noté encore  $\sigma^*$ de $\mathbf{Top}(\Delta^n,\Delta^{n_1}\times\ldots\times\Delta^{n_k})$. 
Par ailleurs, $\tilde\gamma_1^{n_1}\times\cdots\times\tilde\gamma_k^{n_k}\in
\mathbf{Top}(\Delta^{n_1}\times\ldots\times\Delta^{n_k},X^n)$, donc 
$(\tilde\gamma_1^{n_1}\times\cdots\times\tilde\gamma_k^{n_k})\circ\sigma^*\in\mathbf{Top}(\Delta^n,X^n)$. 

\begin{lem}\label{lem:c:1210}
Les applications $\mathrm{Ens}_n^k\to\mathbf{Top}(\Delta^n,X^n)$ et $\widetilde{\mathrm{Ens}}_n^k\to\mathbf{Top}(\Delta^n,X^n)$
données respectivement par $((n_1,\ldots,n_k),\sigma)\mapsto
(\tilde\gamma_1^{(n_1)}\times\cdots\times\tilde\gamma_k^{(n_k)})\circ\sigma^*$ et 
$(v,\sigma)\mapsto \tilde\gamma^{(n)}\circ c(v,\sigma)$ sont telles que le diagramme
$$
\xymatrix{\mathrm{Ens}_n^k \ar[rd]\ar^{\eqref{nllebij:1210}}[rr]& & \widetilde{\mathrm{Ens}}_n^k\ar[ld]\\ 
& \mathbf{Top}(\Delta^n,X^n)& }
$$
est commutatif. 
\end{lem}

\begin{proof} Soit $((n_1,\ldots,n_k),\sigma)\in\mathrm{Ens}_n^k$ et $(v,\sigma)\in\widetilde{\mathrm{Ens}}_n^k$
son image par \eqref{nllebij:1210}. Posons $\delta_m:=(1,\ldots,1)\in\mathbb R^m$ pour $m\geq1$. 
Alors pour $(x_1,\ldots,x_n)\in\Delta^n$, on a 
\begin{align*}
&\tilde\gamma^{(n)}\circ c(v,\sigma)(x_1,\ldots,x_n)
\\ & \scriptstyle{=
\tilde\gamma^{(n)}\Big((x_{\sigma(1)},\ldots,x_{\sigma(n_1)})/k,(\delta_{n_2}+(x_{\sigma(n_1+1)},\ldots,x_{\sigma(n_1+n_2)}))/k,\ldots,
((k-1)\delta_{n_k}+(x_{\sigma(n_1+\cdots+n_{k-1}+1)},\ldots,x_{\sigma(n)}))/k\Big)}
\\ & 
\scriptstyle{=    
\Big(
\tilde\gamma^{n_1}\Big((x_{\sigma(1)},\ldots,x_{\sigma(n_1)})/k\Big),
\tilde\gamma^{n_2}\Big((\delta_{n_2}+(x_{\sigma(n_1+1)},\ldots,x_{\sigma(n_1+n_2)}))/k\Big),
\ldots,
\tilde\gamma^{n_k}\Big(((k-1)\delta_{n_k}+(x_{\sigma(n_1+\cdots+n_{k-1}+1)},\ldots,x_{\sigma(n)}))/k\Big)
\Big)}
\\ & 
=\Big(
\tilde\gamma_1^{n_1}(x_{\sigma(1)},\ldots,x_{\sigma(n_1)}),
\tilde\gamma_2^{n_2}(x_{\sigma(n_1+1)},\ldots,x_{\sigma(n_1+n_2)})),
\ldots,
\tilde\gamma_k^{n_k}(x_{\sigma(n_1+\cdots+n_{k-1}+1)},\ldots,x_{\sigma(n)}))\Big)
\\ & 
=\Big(
\tilde\gamma_1^{(n_1)}(x_{\sigma(1)},\ldots,x_{\sigma(n_1)}),
\tilde\gamma_2^{(n_2)}(x_{\sigma(n_1+1)},\ldots,x_{\sigma(n_1+n_2)}),
\ldots,
\tilde\gamma_k^{(n_k)}(x_{\sigma(n_1+\cdots+n_{k-1}+1)},\ldots,x_{\sigma(n)})\Big)
\\ & 
=(\tilde\gamma_1^{(n_1)}\times\cdots\times\tilde\gamma_k^{(n_k)})\circ\sigma^*(x_1,\ldots,x_n), 
\end{align*}
où la première égalité suit de ce que $c(v,\sigma)$ est l'application de $\Delta^n$ dans lui-même donnée par 
\begin{align*}
&(x_1,\ldots,x_n)\mapsto 
\\ & 
\Big((x_{\sigma(1)},\ldots,x_{\sigma(n_1)})/k,(\delta_{n_2}+(x_{\sigma(n_1+1)},\ldots,x_{\sigma(n_1+n_2)}))/k,\ldots,
((k-1)\delta_{n_k}+(x_{\sigma(n_1+\cdots+n_{k-1}+1)},\ldots,x_{\sigma(n)}))/k\Big), 
\end{align*}
la deuxième égalité suit de la combinaison du fait que $\tilde\gamma^{(n)}$ est une restriction de 
$\tilde\gamma^{n}$ et de l'égalité $\tilde\gamma^{n}=\tilde\gamma^{n_1}\times\cdots\times\tilde\gamma^{n_k}$, 
la troisième égalité suit de l'identité $\tilde\gamma((i-1+x)/k)
=\tilde\gamma_i(x)$ pour $x\in[0,1]$ et $i\in[\![1,k]\!]$, la quatrième égalité suit des relations 
$(x_{\sigma(1)},\ldots,x_{\sigma(n_1)})\in\Delta^{n_1}$, etc., 
$(x_{\sigma(n_1+\cdots+n_{k-1}+1)},\ldots,x_{\sigma(n)})\in\Delta^{n_k}$, elles-mêmes conséquences de 
$\sigma\in\mathfrak S_{n_1,\ldots,n_k}$, la dernière égalité suit de la définition de $\sigma^*$
(cf. section \ref{calcul:2809}). On a donc $\tilde\gamma^{(n)}\circ c(v,\sigma)=(\tilde\gamma_1^{(n_1)}\times\cdots\times\tilde\gamma_k^{(n_k)})
\circ\sigma^*$ (égalité dans $\mathbf{Top}(\Delta^n,X^n)$).
\end{proof}

\begin{prop}\label{lem:0310}
Si $a_1,\ldots,a_{k+1}\in X$ et $\tilde\gamma_i\in\mathrm{Chem}(a_i,a_{i+1})$ pour $i\in[\![1,k]\!]$. Alors 
$$
(\mathrm{div}_n^k)^*((\tilde\gamma_k*\cdots*\tilde\gamma_1)^{(n)})
=\sum_{\substack{n_1,\ldots,n_k\geq 0| \\ n_1+\cdots+n_k=n}}
\sum_{\sigma\in \mathfrak S_{n_1,\ldots,n_k}}
\epsilon(\sigma)(\tilde\gamma_1^{(n_1)}\times\cdots\times\tilde\gamma_k^{(n_k)})\circ\sigma^* 
$$
(égalité dans $C_n(X^n)=\mathbb Z\mathbf{Top}(\Delta^n,X^n)$). 
\end{prop}

\begin{proof} On a 
\begin{align*}
&(\mathrm{div}_n^k)^*((\tilde\gamma_k*\cdots*\tilde\gamma_1)^{(n)})
=(\tilde\gamma_k*\cdots*\tilde\gamma_1)^{(n)}\circ \mathrm{div}_n^k
=\sum_{(\sigma,v)\in\mathrm{Ens}_n^k}
\epsilon(\sigma)(\tilde\gamma_k*\cdots*\tilde\gamma_1)^{(n)}\circ c(\sigma,v)
\\ & 
=\sum_{\substack{n_1,\ldots,n_k\geq 0| \\ n_1+\cdots+n_k=n}}
\sum_{\sigma\in \mathfrak S_{n_1,\ldots,n_k}}
\epsilon(\sigma)(\tilde\gamma_1^{(n_1)}\times\cdots\times\tilde\gamma_k^{(n_k)})\circ\sigma^* 
\end{align*}
où la première égalité suit de la déf. \ref{def:x^*}, la deuxième égalité suit de la déf. \ref{def:div:n:k}, et la troisième suit du 
lemme \ref{lem:c:1210}.  
\end{proof}

\subsubsection{Une égalité dans $C_n(X^n)$}

On fixe $\tilde\gamma\in\mathrm{Chem}(a,b)$ et $\alpha_0,\ldots,\alpha_n\in\mathrm{Chem}(a,a)$. 
On note pour $I\subset [\![0,n]\!]$, $c_I:=(\tilde\gamma*\Asterisk_{i\in I}\tilde\alpha_i)^{(n)}\in C_n(X^n)$. 
En appliquant \eqref{132:2809} à $c_I$ ($X$ étant remplacé par $X^n$ et $k$ par $|I|+1$), on trouve 
\begin{equation}\label{interm:2809}
   c_I-(\mathrm{div}_n^{|I|+1})^*(c_I)=
(\partial_{n,n+1}^*\circ (L^{|I|+1}_{n+1,n})^*+(L^{|I|+1}_{n,n-1})^*\circ\partial_{n-1,n}^*)(c_I) 
 \end{equation}
(relation dans $C_n(X^n)$). 
On a $\partial_{n-1,n}^*(c_I)\in C_{n-1}(Y^{(n)}_{ab})$ par le lemme \ref{lem:1:2}, 
le lemme \ref{lem:images:source:2809} implique alors que $(L^{|I|+1}_{n,n-1})^*\circ\partial_{n-1,n}^*(c_I)\in C_n(Y^{(n)}_{ab})$. 
D'autre part, $(L^{|I|+1}_{n+1,n})^*(c_I)\in C_{n+1}(X^n)$, donc $\partial_{n,n+1}^*\circ (L^{|I|+1}_{n+1,n})^*(c_I)\in
\partial_{n,n+1}^*(C_{n+1}(X^n))$. Ces deux relations et \eqref{interm:2809} impliquent 
$$
\forall I\subset [\![0,n]\!],\quad c_I-(\mathrm{div}_n^{|I|+1})^*(c_I)\in C_n(Y^{(n)}_{ab})+\partial_{n,n+1}^*(C_{n+1}(X^n)) 
$$
(relation dans $C_n(X^n)$).
Cette relation implique $\sum_{I\subset [\![0,n]\!]}(-1)^{|I|}(c_I-(\mathrm{div}_n^{|I|+1})^*(c_I))\in 
C_n(Y^{(n)}_{ab})+\partial_{n,n+1}^*(C_{n+1}(X^n))$
donc 
\begin{equation}\label{interm:bis:2809}
  \sum_{I\subset [\![0,n]\!]}(-1)^{|I|}c_I-
\sum_{I\subset [\![0,n]\!]}(-1)^{|I|}(\mathrm{div}_n^{|I|+1})^*(c_I)\in C_n(Y^{(n)}_{ab})+\partial_{n,n+1}^*(C_{n+1}(X^n)). 
 \end{equation}
(relation dans $C_n(X^n)$). 

\begin{lem}\label{lem:2:50:0412}
On a  $\sum_{I\subset [\![0,n]\!]}(-1)^{|I|}(\mathrm{div}_n^{|I|+1})^*(c_I)=0$ (égalité dans $C_n(X^n)$).  
\end{lem}

\begin{proof}
Pour $\nu_0,\ldots,\nu_{n+1}\geq 0$ avec $\nu_0+\cdots+\nu_{n+1}=n$, posons 
\begin{equation}\label{def:f:nu_i}
f(\nu_0,\ldots,\nu_{n+1}):=\sum_{\sigma\in\mathfrak S_{\nu_0,\ldots,\nu_{n+1}}}\epsilon(\sigma)(\tilde\alpha^{(\nu_0)}\times\cdots\times\tilde\alpha^{(\nu_n)}
    \times\tilde\gamma^{(\nu_{n+1})})\circ\sigma^*\in C_n(X^n). 
\end{equation}
Soit $I\subset [\![0,n]\!]$. Soit $\alpha\mapsto i_\alpha$ l'unique bijection croissante 
$[\![1,|I|]\!]\to I$. On a les égalités  
\begin{align}\label{1211:0310}
    &(\mathrm{div}_n^{|I|+1})^*(c_I)=
    \sum_{\substack{\nu : I\sqcup\{n+1\}\to\mathbb Z_{\geq 0}, \\ \sum_{x\in I\sqcup\{n+1\}}\nu(x)=n}}
    \sum_{\sigma\in \mathfrak S_{\nu(i_1),\ldots,\nu(i_{|I|}),\nu(n+1)}}
    \epsilon(\sigma)(\tilde\alpha_{i_1}^{(\nu(i_1))}\times\cdots\times\tilde\alpha_{i_{|I|}}^{(\nu(i_{|I|}))}
    \times\tilde\gamma^{(\nu(n+1))})\circ\sigma^*\\ & \nonumber 
    = \sum_{\substack{\nu_0,\ldots,\nu_{n+1}\geq 0|\\ \nu_0+\cdots+\nu_{n+1}=n,\\ \nu_{|[\![0,n]\!]-I}=0}}
    \sum_{\sigma\in\mathfrak S_{\nu_0,\ldots,\nu_{n+1}}}\epsilon(\sigma)(\tilde\alpha^{(\nu_0)}\times\cdots\times\tilde\alpha^{(\nu_n)}
    \times\tilde\gamma^{(\nu_{n+1})})\circ\sigma^*
    = \sum_{\substack{\nu_0,\ldots,\nu_{n+1}\geq 0|\\ \nu_0+\cdots+\nu_{n+1}=n,\\ \nu_{|[\![0,n]\!]-I}=0}}
    f(\nu_0,\ldots,\nu_{n+1})
\end{align}
dans $C_n(X^n)$
où la première égalité suit de la proposition \ref{lem:0310}, et la deuxième égalité utilise la bijection entre applications 
$I\sqcup\{n+1\}\to\mathbb Z_{\geq 0}$ et applications $[\![0,n+1]\!]\to\mathbb Z_{\geq 0}$ nulles sur $[\![0,n]\!]-I$
fournie par l'extension par la fonctions nulle, et la troisième égalité suit de \eqref{def:f:nu_i}. 

Alors 
\begin{align}\label{12:33:0310}
    & \nonumber \sum_{I\subset [\![0,n]\!]}(-1)^{|I|}(\mathrm{div}_n^{|I|+1})^*(c_I)
    = \sum_{I\subset [\![0,n]\!]}\sum_{\substack{\nu_0,\ldots,\nu_{n+1}\geq 0|\\ \nu_0+\cdots+\nu_{n+1}=n,\\ \nu_{|[\![0,n]\!]-I}=0}}
    (-1)^{|I|}f(\nu_0,\ldots,\nu_{n+1})
    \\ & \nonumber= \sum_{\substack{\nu_0,\ldots,\nu_{n+1}\geq 0|\\ \nu_0+\cdots+\nu_{n+1}=n}}
    \sum_{\{x\in[\![0,n]\!]|\nu_x\neq 0\}\subset I\subset [\![0,n]\!]}
    (-1)^{|I|}f(\nu_0,\ldots,\nu_{n+1})
    \\ & = \sum_{\substack{\nu_0,\ldots,\nu_{n+1}\geq 0|\\ \nu_0+\cdots+\nu_{n+1}=n}}
    f(\nu_0,\ldots,\nu_{n+1})
    \sum_{\{x\in[\![0,n]\!]|\nu_x\neq 0\}\subset I\subset [\![0,n]\!]}
    (-1)^{|I|}
\end{align}
où la première égalité suit de \eqref{1211:0310}, la deuxième égalité suit de l'équivalence entre les conditions  
$\nu_{|[\![0,n]\!]-I}=0$ et $I\supset\{x\in[\![0,n]\!]|\nu_x\neq 0\}$ et la troisième égalité est une factorisation.  

Pour tout $(\nu_0,\ldots,\nu_{n+1})\in\mathbb Z_{\geq 0}^{n+1}$ tel que $\nu_0+\cdots+\nu_{n+1}=n$, on a
\begin{equation}\label{12:32:0310}
\sum_{\{x\in[\![0,n]\!]|\nu_x\neq 0\}\subset I\subset [\![0,n]\!]}(-1)^{|I|}
=(-1)^{|\{x\in[\![0,n]\!]|\nu_x\neq 0\}|}\sum_{J\subset \{x\in[\![0,n]\!]|\nu_x=0\}}(-1)^{|J|}
=0
\end{equation}
où la première égalité suit de la bijection entre l'ensemble des $I$ tels que  $\{x\in[\![0,n]\!]|\nu_x\neq 0\}\subset I\subset [\![0,n]\!]$
et l'ensemble des parties $J$ de $\{x\in[\![0,n]\!]|\nu_x=0\}$ fournie par $J\mapsto J\cup \{x\in[\![0,n]\!]|\nu_x\neq 0\}$, et la seconde
égalité suit de l'identité $\sum_{X\subset E}(-1)^{|X|}=0$ pour tout ensemble fini non vide $E$, ainsi que de  
$\{x\in[\![0,n]\!]|\nu_x=0\}\neq \emptyset$, qui résulte de $(\nu_0,\ldots,\nu_{n+1})\in\mathbb Z_{\geq 0}^{n+1}$ et  
$\nu_0+\cdots+\nu_{n+1}=n$. 
En combinant \eqref{12:32:0310} et \eqref{12:33:0310}, on obtient $\sum_{I\subset [\![0,n]\!]}(-1)^{|I|}(\mathrm{div}_n^{|I|+1})^*(c_I)=0$. 
\end{proof}

\subsubsection{Démonstration de (b) du théorème \ref{thm:ppal}}

Le lemme \ref{lem:2:50:0412} et \eqref{interm:bis:2809} impliquent 
$$
\sum_{I\subset [\![0,n]\!]}(-1)^{|I|}c_I\in C_n(Y^{(n)}_{ab})+\partial_{n,n+1}^*(C_{n+1}(X^n))
$$
(relation dans $C_n(X^n)$)
c'est-à-dire  
\begin{equation}\label{int:2248:2809}
\sum_{I \subset[\![0,n]\!]}(-1)^{|I|}(\tilde\gamma*\Asterisk_{i\in I}\tilde\alpha_i)^{(n)} \in C_n(Y^{(n)}_{ab})+\partial_{n,n+1}^*(C_{n+1}(X^n))
\end{equation}
(relation dans $C_n(X^n)$). Chaque $c_I=(\tilde\gamma*\Asterisk_{i\in I}\tilde\alpha_i)^{(n)}$ appartient au sous-espace 
$Z_n(X^n,Y^{(n)}_{ab})=\{c\in C_n(X^n)|\partial^*_{n-1,n}(c)\in C_{n-1}(Y^{(n)}_{ab})\}\subset C_n(X^n)$, donc \eqref{int:2248:2809} 
peut être vue comme une relation dans $Z_n(X^n,Y^{(n)}_{ab})$. Elle implique la relation 
\begin{equation}\label{interm:3:222:32:2809}
\sum_{I \subset[\![0,n]\!]}(-1)^{|I|}[(\tilde\gamma*\Asterisk_{i\in I}\tilde\alpha_i)^{(n)}]=0 
\end{equation}
(égalité dans $\mathrm H_n(X^n,Y^{(n)}_{ab})=Z^n(X^n,Y^{(n)}_{ab})/B_n(X^n,Y^{(n)}_{ab})$, avec $B_n(X^n,Y^{(n)}_{ab})
=C_n(Y^{(n)}_{ab})+\partial_{n,n+1}^*(C_{n+1}(X^n))$). 

Pour $I \subset[\![0,n]\!]$, on a 
$$
[(\tilde\gamma*\Asterisk_{i\in I}\tilde\alpha_i)^{(n)}]=
F_n([\tilde\gamma*\asterisk_{i\in I}\tilde\alpha_i])
=F_n(\gamma \cdot \prod_{i\in I}\alpha_i)
$$
(égalité dans $\mathrm H_n(X^n,Y^{(n)}_{ab})$) où la première égalité suit de \eqref{***:2809} et deuxième suit de (\eqref{eq:2106:2505}). 

En combinant cette égalité avec \eqref{interm:3:222:32:2809}, on en déduit l'égalité souhaitée
$$
\sum_{I \subset[\![0,n]\!]}(-1)^{|I|}F_n(\gamma \cdot \prod_{i\in I}\alpha_i)=0
$$
(égalité dans $\mathrm H_n(X^n,Y^{(n)}_{ab})$).

\section{Lien avec l'isomorphisme de Beilinson}\label{sect:3:2512}

Le but de cette section est la démonstration de la proposition \ref{prop:beil:2012}, qui relie l'application $\overline F^{(n)}_{xy}$ obtenue dans 
le théorème \ref{thm:ppal} avec l'isomorphisme \eqref{iso:beil} de Beilinson. On rappelle la construction de cet isomorphisme en 
section \ref{sect:31:2512}, puis on montre la proposition \ref{prop:beil:2012} en section \ref{subsect:3:2:2512}.  

\subsection{Rappels sur l'isomorphisme de Beilinson}\label{sect:31:2512} 

Soit $M$ une variété différentiable connexe, ayant le type d'homotopie d'un CW-complexe fini et $n\geq1$. Dans \cite{DG}, \S3.3
(voir aussi \cite{BGF}, p. 251), on associe à chaque couple $(x,y)$ d'éléments de $M$ un complexe de faisceaux de $\mathbb Q$-espaces 
vectoriels $_y \mathcal K_x\langle n\rangle$ sur $M^n$ et une application linéaire surjective 
$\mathbb H^\bullet(M^n,_x \mathcal K_x\langle n\rangle)\to\mathbb Q$. On a les isomorphismes
de $\mathbb Q$-espaces vectoriels
$$
\mathrm H^\bullet(M^n,Y^{(n)}_{yx};\mathbb Q) \simeq \mathbb H^\bullet(M^n,_y \mathcal K_x\langle n\rangle) \text{ si } x \neq y, 
$$
(\cite{BGF}, deux lignes avant (3.282)) et
$$
\mathrm H^\bullet(M^n,Y^{(n)}_{xx};\mathbb Q) \simeq \mathrm{Ker}(\mathbb H^\bullet(M^n,_x \mathcal K_x\langle n\rangle)\to\mathbb Q) , 
$$
(\cite{BGF}, (3.284)), où $\mathrm H^\bullet(-,-;\mathbb Q)$ désigne l'homologie singulière relative à coefficients dans 
$\mathbb Q$ des paires d'espaces topologiques et $\mathbb H^\bullet$ l'hypercohomologie des complexes de faisceaux. 

Le théorème de Beilinson (Proposition 3.4 de \cite{DG}, ou Theorem 3.298 de \cite{BGF}) dit qu'il y a un isomorphisme
de $\mathbb Q$-espaces vectoriels
$$
\underline\beta_{yx}^{(n)} : \mathbb H^\bullet(M^n,_y \mathcal K_x\langle n\rangle) \to(\mathbb Q\pi_1(x,y)/\mathbb Q\pi_1(x,y)(\mathbb Q\pi_1(x))_+^{n+1})^*
$$
s'insérant pour $y=x$ dans le diagramme commutatif
$$
\xymatrix{\mathbb H^\bullet(M^n,_x \mathcal K_x\langle n\rangle)\ar[rr]\ar[rd]&&
(\mathbb Q\pi_1(x)/(\mathbb Q\pi_1(x))_+^{n+1})^*
\ar[ld]\\&\mathbb Q&}
$$
l'application $(\mathbb Q\pi_1(x)/(\mathbb Q\pi_1(x))_+^{n+1})^*\to\mathbb Q$ étant duale de l'application 
$\mathbb Q\to\mathbb Q\pi_1(x)/(\mathbb Q\pi_1(x))_+^{n+1}$ induite par $1\mapsto 1$. 

On en déduit pour tout $(x,y)$ une application linéaire 
\begin{equation}\label{BEIL:0912}
\beta^{(n)}_{yx} : \mathrm H^\bullet(M^n,Y^{(n)}_{yx};\mathbb Q)\to (\mathbb Q\pi_1(x,y)/\mathbb Q\pi_1(x,y)
(\mathbb Q\pi_1(x))_+^{n+1})^*, 
\end{equation}
qui est un isomorphisme si $y\neq x$, et qui induit un isomorphisme  
$$
\mathrm H^\bullet(M^n,Y^{(n)}_{xx};\mathbb Q)\stackrel{\sim}{\to} \mathrm{Ker}((\mathbb Q\pi_1(x)/
(\mathbb Q\pi_1(x))_+^{n+1})^*\to\mathbb Q)
$$
si $y=x$. 

\subsection{Relation du théorème \ref{thm:ppal} avec l'isomorphisme de Beilinson}\label{subsect:3:2:2512}

Notons $\theta$ l'involution de $M^n$ donnée par $(x_1,\ldots,x_n)\mapsto (x_n,\ldots,x_1)$. L'image de $Y^{(n)}_{yx}$ par cette involution 
est $Y^{(n)}_{xy}$, donc elle induit un isomorphisme $\theta^* : \mathrm H^n(M^n,Y^{(n)}_{yx};\mathbb Q)\to \mathrm H^n(M^n,Y^{(n)}_{xy};
\mathbb Q)$. Rappelons le couplage entre homologie et cohomologie relatives, qui induit une application linéaire 
$\mathrm{can} : \mathrm H^n(M^n,Y^{(n)}_{xy};\mathbb Q)\to \mathrm H_n(M^n,Y^{(n)}_{xy})_{\mathbb Q}^*$, où pour
$A$ un $\mathbb Z$-module, on note $A_{\mathbb Q}^*:=\mathrm{Hom}_{\mathbb Z}(A,\mathbb Q)$. Pour $f : A\to B$ morphisme de $\mathbb Z$-modules, 
on note aussi $f_{\mathbb Q}^* : B_{\mathbb Q}^*\to A_{\mathbb Q}^*$ le morphisme induit. 

\begin{prop}\label{prop:beil:2012}
L'application linéaire $\beta^{(n)}_{yx}$ (cf. \eqref{BEIL:0912}) est au signe près, la composée du dual $(\overline F^{(n)}_{xy})_{\mathbb Q}^*$ 
de l'application linéaire $\overline F^{(n)}_{xy}$ (cf. théorème \ref{thm:ppal}), de $\mathrm{can}$ et de $\theta^*$. Précisément, on a  
$$
\beta^{(n)}_{yx}=(-1)^{n+1}(\overline F^{(n)}_{xy})_{\mathbb Q}^*\circ \mathrm{can}\circ\theta^*. 
$$
\end{prop}

\begin{proof}
Dans \cite{BGF}, on construit un morphisme de complexes de faisceaux $\mathrm{nat} : _y \tilde{\mathcal K}_x\langle n\rangle
\to _y \mathcal K_x\langle n\rangle$ ((3.282) et cinq lignes avant cette équation) ; si $y\neq x$, ce morphisme est l'identité de 
$_y \mathcal K_x\langle n\rangle$ (cf. {\it loc. cit.,} 6 lignes avant (3.282)). Dans {\it loc. cit.,} on construit un isomorphisme 
d'espaces vectoriels $iso_{\mathrm{BGF}}^{yx} : \mathbb H^n(M^n,_y \tilde{\mathcal K}_x\langle n\rangle)
\to \mathrm H^n(M^n,Y^{(n)}_{yx};\mathbb Q)$ (Lemma 3.281). L'application \eqref{BEIL:0912}
est alors donnée par la composition 
$$
\mathrm H^n(M^n,Y^{(n)}_{yx};\mathbb Q)
\stackrel{(iso_{\mathrm{BGF}}^{yx})^{-1}}{\to}
\mathbb H^n(M^n,_y \tilde{\mathcal K}_x\langle n\rangle)
\to\mathbb H^n(M^n,_y \mathcal K_x\langle n\rangle)
\stackrel{\underline\beta^{(n)}_{yx}}{\to}
(\mathbb Q\pi_1(x,y)/\mathbb Q\pi_1(x,y)(\mathbb Q\pi_1(x))_+^{n+1})^*
$$
L'énoncé est alors équivalent à la commutativité du diagramme suivant 
$$
\xymatrix{
\mathbb H^n(M^n,_y \tilde{\mathcal K}_x\langle n\rangle)
\ar^{\mathrm{nat}^*}[r]\ar_{iso_{\mathrm{BGF}}^{yx}}[d]&
\mathbb H^n(M^n,_y \mathcal K_x\langle n\rangle)
\ar^{\!\!\!\!\!\!\!\!\!\!\!\!\!\!\!\!\!\!\!\!\!\!\!\underline\beta^{(n)}_{yx}}[r]&
(\mathbb Q\pi_1(x,y)/\mathbb Q\pi_1(x,y)(\mathbb Q\pi_1(x))_+^{n+1})^*
\\
\mathrm H^n(M^n,Y^{(n)}_{yx};\mathbb Q)\ar_{\theta^*}[r]&
\mathrm H^n(M^n,Y^{(n)}_{xy};\mathbb Q)\ar_{(-1)^{n+1}\mathrm{can}}[r]&
\mathrm H_n(M^n,Y^{(n)}_{xy})_{\mathbb Q}^*\ar_{(\overline F^{(n)}_{xy})^*}[u]}
$$
laquelle est équivalente à l'énoncé suivant : 
\begin{align}\label{a:montrer:beil}
    & \forall c\in\mathrm H^n(M^n,Y^{(n)}_{xy};\mathbb Q),\quad \forall a\in
    \mathbb Q\pi_1(x,y)/(\mathbb Q\pi_1(x,y)(\mathbb Q\pi_1(x))_+^{n+1}),
    \\ & \nonumber \langle\underline\beta^{(n)}_{yx}\circ\mathrm{nat}^*\circ (iso_{\mathrm{BGF}}^{yx})^{-1}
    \circ(\theta^*)^{-1}(c),a
    \rangle
    =(-1)^{n+1}\langle c,\overline F^{(n)}_{xy}(a)\rangle_{hom}, 
\end{align}
où $\langle-,-\rangle$ est le couplage $V^*\times V\to\mathbb Q$ associé à un $\mathbb Q$-espace vectoriel $V$
et $\langle-,-\rangle_{hom}$ est le couplage naturel $\mathrm H^n(X,Y;\mathbb Q)\times\mathrm H_n(X,Y)
\to\mathbb Q$, que par linéarité il suffit de vérifier pour $a=[\gamma]$, où $\gamma\in\pi_1(x,y)$. 

Rappelons quelques constructions de \cite{BGF}. Soit $\tilde{\mathbf C}:=((\tilde C^{p,q})_{p,q\geq0},d',d'')$  
le bicomplexe tel que $\tilde C^{p,q}:=\oplus_{I \subset [\![1,n]\!]||I|=n-p}C^q(X^I)$, où 
$C^q(X):=\mathrm{Hom}_{\mathbb Q}(C_q(X),\mathbb Q)$, où $d''$ est la somme sur $I,q$ des opérateurs de cobord  
$C^q(X^I)\to C^{q+1}(X^I)$ et où $d''$ est la somme sur les couples $(I,J)$ avec $I\supset J$ et $|I-J|=1$ 
des applications $\epsilon(I,J)\delta_{I,J}^* : C^q(X^I)\to C^q(X^J)$, où $\delta_{I,J} : X^J\to X^I$ 
est le morphisme donné par \cite{BGF}, formule après (3.287) et $\epsilon(I,J) \in \{\pm1\}$ est donné par \cite{BGF}, 
formule avant (3.278). Soit $\mathbf C:=((C^{p,q})_{p,q\geq 0},d',d'')$ le bicomplexe quotient de $\tilde{\mathbf C}$  
donné par $C^{p,q}=\tilde C^{p,q}$ si $q<n$, $C^{p,n}=0$ pour tout $p\geq0$. 

D'après \cite{BGF}, Lemma 3.289, on a un isomorphisme $\mathrm H^\bullet(\mathrm{Tot}(\mathbf C))\simeq 
\mathbb H^n(M^n,_y \mathcal K_x\langle n\rangle)$. On montre que cet isomorphisme s'insère dans carré commutatif 
dont les morphismes verticaux sont des isomorphismes
\begin{equation}\label{comm:2012}
\xymatrix{\mathrm H^\bullet(\mathrm{Tot}\tilde{\mathbf C})\ar^{\mathbf{nat}}[r]\ar^{\sim}_{\tilde i_{\mathbb C}}[d]&
\mathrm H^\bullet(\mathrm{Tot}\mathbf C)\ar_{\sim}^{i_{\mathbb C}}[d]\\
\mathbb H^n(M^n,_y \tilde{\mathcal K}_x\langle n\rangle)
\ar_{\mathrm{nat}^*}[r]&
\mathbb H^n(M^n,_y \mathcal K_x\langle n\rangle)}
\end{equation}
Alors l'application composée
$$
\mathrm H^n(\mathrm{Tot}\mathbf C)
\stackrel{i_{\mathbf C}}{\to} \mathbb H^n(M^n,_y \mathcal K_x\langle n\rangle)
\stackrel{\underline\beta^{(n)}_{yx}}{\to}
(\mathbb Q\pi_1(x,y)/\mathbb Q\pi_1(x,y)(\mathbb Q\pi_1(x))_+^{n+1})^*
$$
est telle que pour $\omega=(\omega_I)_{\emptyset\neq I\subset[\![1,n]\!]}\in\mathrm Z^n(\mathrm{Tot}\mathbf C)$
avec $\omega_I\in C^{n-|I|}(X^I)$, et $\gamma\in \pi_1(x,y)$, on a
\begin{equation}\label{pairing:BGF}
\langle\underline\beta_{yx}^{(n)}\circ i_{\mathbf C}([\omega]),[\gamma]
\rangle
=\sum_{\emptyset\neq I\subset[\![1,n]\!]}(-1)^{(1/2)(|I|-1)(|I|-2)+n|I|}
\epsilon(I)
\omega_I(\tilde\gamma^{(I)}), 
\end{equation}
(cf. \cite{BGF}, (3.292)), où $\epsilon(I)$ est donné par \cite{BGF}, (3.278) et 
$\gamma^{(I)}\in \mathbf{Top}(\Delta^{|I|},M^I)\subset C_{|I|}(M^I)$ donné par 
$\Delta^{|I|}\ni (t_1,\ldots,t_{|I|})\mapsto (I\ni i\mapsto \tilde\gamma(t_{\kappa(i)}))\in M^I$, avec 
$\kappa$ l'unique bijection croissante $I\to[\![1,|I|]\!]$. 

Pour tout $i \in [\![1,n]\!]$ et $p \geq 0$, l'application 
$\delta^*_{[\![1,n]\!],[\![1,n]\!]-\{i\}} : C^p(X^n)\to C^p(X^{[\![1,n]\!]-\!{i\!}})$ est une composition 
$C^p(X^n)\to C^p(Y^{(n)}_{xy})\to C^p(X^{[\![1,n]\!]-\{i\}})$ donc 
$\mathrm{Ker}(C^p(X^n)\to C^p(Y^{(n)}_{xy})) \subset \mathrm{Ker}(\oplus_{i\in[\![1,n]\!]}: \delta^*_{[\![1,n]\!],[\![1,n]\!]-\{i\}} : 
C^p(X^n)\to \oplus_{i \in [\![1,n]\!]} C^p(X^{[\![1,n]\!]-\{i\}}))$. 

On a donc pour chaque $p \geq 0$ une application linéaire $\tilde\mu^p : \mathrm{Ker}(C^p(X^n)\to C^p(Y^{(n)}_{xy}))
\to \mathrm{Tot}^p(\tilde{\mathbf C})$ donnée par $\mathrm{Ker}(C^p(X^n)\to C^p(Y^{(n)}_{xy})) \ni c\mapsto ((p,0)\mapsto c, 
(p,0)\neq (p,'q')\mapsto 0)$, qui définit un morphisme de complexes 
\begin{equation}\label{MORCOMPL} 
\tilde\mu^\bullet : \mathrm{Ker}(C^\bullet(X^n)\to C^\bullet(Y^{(n)}_{xy}))\to \mathrm{Tot}^\bullet(\tilde{\mathbf C}).
\end{equation}
La cohomologie du complexe source de \eqref{MORCOMPL} est la cohomologie singulière relative
$\mathrm H^\bullet(X^n,Y^{(n)}_{xy};\mathbb Q)$. On déduit du morphisme de complexes \eqref{MORCOMPL} un morphisme en cohomologie  
$$
\mathrm H^\bullet(\tilde\mu^\bullet) : \mathrm H^\bullet(X^n,Y^{(n)}_{xy};\mathbb Q)\to \mathrm H^\bullet(\mathrm{Tot}(\tilde{\mathbf C}))
$$
dont on vérifie qu'il satisfait 
\begin{equation}\label{STEP:2012}
    \mathrm H^\bullet(\tilde\mu^\bullet)
    =(\tilde i_{\mathbf C})^{-1} \circ (iso_{\mathrm{BGF}})^{-1} \circ (\theta^*)^{-1}.
\end{equation}

Montrons alors \eqref{a:montrer:beil}. Soit $\gamma\in\pi_1(x,y)$, $c\in \mathrm H^n(M^n,Y^{(n)}_{xy};\mathbb Q)$. 
Soit $\tilde c\in \mathrm Z^n(\mathrm{Ker}(C^\bullet(X^n)\to C^\bullet(Y^{(n)}_{xy})))$ un représentant de $c$. 
Alors 
\begin{align*}
    & \langle\underline\beta^{(n)}_{yx}\circ\mathrm{nat}^*\circ (iso_{\mathrm{BGF}}^{yx})^{-1}
    \circ(\theta^*)^{-1}(c),[\gamma]
    \rangle
    =\langle\underline\beta^{(n)}_{yx}\circ i_{\mathbf C}\circ \mathbf{nat}\circ \tilde i_{\mathbf C}^{-1}
    \circ (iso_{\mathrm{BGF}}^{yx})^{-1}
    \circ(\theta^*)^{-1}(c),[\gamma]
    \rangle
    \\ & =\langle\underline\beta^{(n)}_{yx}\circ i_{\mathbf C}\circ \mathbf{nat}\circ 
    \mathrm H^n(\tilde\mu^\bullet)(c),[\gamma]
    \rangle
   =\langle\underline\beta^{(n)}_{yx}\circ i_{\mathbf C}\circ \mathbf{nat}\circ 
    \mathrm [\tilde\mu^n(\tilde c)],[\gamma]
    \rangle
   =\langle\underline\beta^{(n)}_{yx}\circ i_{\mathbf C}\circ  
    \mathrm [\mu^n(\tilde c)],[\gamma]
    \rangle
    \\ & =(-1)^{(n-1)(n-2)/2+n^2+n(n+1)/2} \tilde c(\tilde\gamma^{(n)})
    =(-1)^{n+1}\langle c,\overline F^{(n)}_{xy}([\gamma])
    \rangle_{hom}
\end{align*}
où $\mu^n$ est la composée de $\tilde\mu^n$ et de la projection $\tilde{\mathbf C}^n\to\mathbf C^n$ ; 
la première égalité suit de la commutativité de \eqref{comm:2012}, la deuxième égalité suite de \eqref{STEP:2012}, 
la troisième égalité suit de la définition de $\mathrm H^n(\tilde\mu^\bullet)$, la quatrième égalité suit de ce que
$\mathbf{nat}$ est la version cohomologiquee de la projection $\tilde{\mathbf C}^n\to\mathbf C^n$, la cinquième égalité suit de 
l'égalité \eqref{pairing:BGF}, dans laquelle seule la contribution de $I=[\![1,n]\!]$ est non-triviale, la dernière égalité suit 
de \eqref{***:2809}.  
\end{proof}

La démonstration de la proposition \ref{prop:beil:2012} est illustrée par le diagramme suivant
$$
\xymatrix{
\mathrm H^\bullet(\mathrm{Tot}\tilde{\mathbf C})\ar^{\mathbf{nat}}[r]\ar^{\sim}_{\tilde i_{\mathbb C}}[d]&
\mathrm H^\bullet(\mathrm{Tot}\mathbf C)\ar_{\sim}^{i_{\mathbb C}}[d]\ar^{\underline\beta^{(n)}_{yx}\circ i_{\mathbb C}}[rd]&\\
\mathbb H^n(M^n,_y \tilde{\mathcal K}_x\langle n\rangle)
\ar^{\mathrm{nat}^*}[r]\ar_{iso_{\mathrm{BGF}}^{yx}}[d]&
\mathbb H^n(M^n,_y \mathcal K_x\langle n\rangle)
\ar^{\!\!\!\!\!\!\!\!\!\!\!\!\!\!\!\!\!\!\!\underline\beta^{(n)}_{yx}}[r]&
(\mathbb Q\pi_1(x,y)/\mathbb Q\pi_1(x,y)(\mathbb Q\pi_1(x))_+^{n+1})^*
\\
\mathrm H^n(M^n,Y^{(n)}_{yx};\mathbb Q)\ar_{\theta^*}[r]&
\mathrm H^n(M^n,Y^{(n)}_{xy};\mathbb Q)\ar_{\mathrm{can}}[r]&
\mathrm H_n(M^n,Y^{(n)}_{xy})_{\mathbb Q}^*\ar_{(\overline F^{(n)}_{xy})^*}[u]}
$$

\section{Construction de transformations naturelles}\label{sect:4:2512}


On note $\mathbf{Top}_2$ la catégorie des espaces topologiques munis d’un couple de points marqués 
(i.e. d'un couple de morphismes de source l'objet initial $*$). Pour $(X,a,b)$ un objet de $\mathbf{Top}_2$, 
le couple $(\pi_1(a,b),\pi_1(a))$ est un torseur à droite, à savoir un couple $(T,G)$ avec $T$ un ensemble et $G$ un groupe, 
munis d'une action à droite libre et transtive de $G$ sur $T$. La correspondance $(X,a,b)\mapsto (\pi_1(a,b),\pi_1(a))$ définit 
un foncteur $\mathbf{Top}_2\to\mathbf{TorDt}$ avec $\mathbf{TorDt}$ la catégorie des torseurs à droite. 


\begin{defn}
Pour $(X,x,y)$ un objet de $\mathbf{Top}_2$, on note $\mathbb Z\pi_1(X,x,y)$ le $\mathbb Z$-module libre sur  
$\pi_1(X,x,y)$.
\end{defn}

\begin{lem}
L'application $(X,a,b)\mapsto \mathbf F_n(X,a,b)$ 
définit un foncteur covariant de $\mathbf{Top}_2$ vers la catégorie $\mathbf{Ab}$ des groupes abéliens. 
\end{lem}

\begin{proof}
Il s'agit de la composition du foncteur $\mathbf{Top}_2\to\mathbf{TorDt}$ envoyant 
$(X,a,b)$ vers $(\pi_1(a,b), \pi_1(a))$ et du foncteur $\mathbf{TorDt}\to\mathbf{Ab}$ envoyant $(T,G)$
vers $\mathbb ZT/(\mathbb ZT)(\mathbb ZG)_+^{n+1}$.
\end{proof}


\begin{defn}
Pour $\underline X=(X,a,b)$ un objet de $\mathbf{Top}_2$, on note $Y_{\underline X}(n):=Y_{ab}^{(n)}$ 
(cf. \eqref{ref:Y:2212}). 
\end{defn}

\begin{lem}
L'application $(X,a,b)\mapsto (X^n,Y_{\underline X}(n))$ est un foncteur covariant $\mathbf{Top}_2\to\mathbf{Paires}$. 
\end{lem}

\begin{proof}
On vérifie que si $\underline X=(X,a,b)$ et $\underline X'=(X',a',b')$ sont des objets de $\mathbf{Top}_2$ et si 
$f : X\to X'$ est un morphisme
dans $\mathbf{Top}$ tel que $f(a)=a'$ et $f(b)=b'$, alors $f(Y_{\underline X}(n))\subset Y_{\underline X'}(n)$. 
\end{proof}

\begin{lemdef}
Pour $\underline X\in\mathbf{Top}_2$, on définit $\mathbf H_n(\underline X):=\mathrm{H}_n(X^n,Y_{\underline X}(n))$. 
L'application $\underline X\mapsto \mathbf H_n(\underline X)$ définit un foncteur 
covariant $\mathbf H_n : \mathbf{Top}_2\to\mathbf{Paires}$. 
\end{lemdef}

\begin{proof}
Provient de l'identification de $\mathbf H_n$ avec la composition du foncteur $(X,a,b)\mapsto (X^n,Y_{\underline X}(n))$ avec le foncteur 
homologie relative. 
\end{proof}


Un corollaire du théorème \ref{thm:ppal} est: 
\begin{thm}\label{thm:transfo:nat}
Pour $n\geq0$ et $\underline X$ un objet de $\mathbf{Top}_2$, le morphisme de groupes 
$\mathbb Z\pi_1(a,b)\to\mathbf H_n(\underline X)$, $\gamma\mapsto [\tilde\gamma^n]$
induit un morphisme de groupes $\nu^{(n)}_{\underline X} : 
\mathbf F_n(\underline X)\to\mathbf H_n(\underline X)$. La correspondance 
$\underline X\mapsto \nu^{(n)}_{\underline X}$ est une transformation naturelle de $\mathbf F_n$ vers $\mathbf H_n$. 
\end{thm}


\medskip\noindent{\bf Remerciements.} Le travail de B.E. a bénéficié du soutien du projet ANR “Project HighAGT ANR20-CE40-0016”.


\begin{thebibliography}{BGFr}


\bibitem[BGFr]{BGF} J. Burgos Gil, J. Fresan, Multiple zeta values: from numbers to motives, preprint 
http://javier.fresan.perso.math.cnrs.fr/mzv.pdf, à paraître dans {\it Clay Mathematics Proceedings.}

\bibitem[DG]{DG} P. Deligne, A. Goncharov, Groupes fondamentaux motiviques de Tate mixte. Ann. Sci. École Norm. Sup. (4) 38 (2005), no. 1, 1–56.


\bibitem[GrH]{Gr} M. Greenberg, J. Harper, Algebraic topology. A first course, Mathematics Lecture Note Series, 58. 
Benjamin/Cummings Publishing Co., Inc., Advanced Book Program, Reading, Mass., 1981.

\bibitem[Ha]{Ha} A. Hatcher, Algebraic topology. New York, Cambridge University Press, 2001. 






\end{thebibliography}
\end{document}